\documentclass[a4paper,12pt,reqno]{amsart}

\usepackage{color}
\usepackage{amsmath}
\usepackage{amssymb}
\usepackage{amsfonts}
\usepackage{graphicx}
\usepackage{mathtools}
\usepackage[colorlinks]{hyperref}
\renewcommand\eqref[1]{(\ref{#1})} %Need with hyperref
\usepackage{booktabs}
\usepackage{lscape}
\usepackage{braket}
\usepackage{mathrsfs}
\usepackage{comment}
\usepackage[shortlabels]{enumitem}
\usepackage{stackrel}
\usepackage{mathtools,bm}
\usepackage{tikz}
\usepackage{tikz-cd}
\usepackage{graphicx}
\usepackage{dsfont}
\usepackage{float}
\usetikzlibrary{decorations.markings}
\usepackage{hyperref}

\graphicspath{ {images/} }
\newcommand{\bbC}{{\Bbb C}}

\newcommand{\bbN}{{\Bbb N}}
\newcommand{\bbR}{{\Bbb R}}

\newcommand{\tr}{\operatorname{tr}}

\title[Functional calculus for Safarov pseudo-differential operators]{Functional calculus for Safarov pseudo-differential operators}

\author[S. G\'omez Cobos]{Santiago G\'omez Cobos}
\address{
	Santiago G\'omez Cobos:
	\endgraf
	Department of Mathematics: Analysis, Logic and Discrete Mathematics
	\endgraf
	Ghent University, Krijgslaan 281, Building S8, B 9000 Ghent
	\endgraf
	Belgium
	\endgraf
	{\it E-mail address} {\rm davidsantiago.gomezcobos@ugent.be}}

\author[M. Ruzhansky]{Michael Ruzhansky}
\address{
	Michael Ruzhansky:
	\endgraf
	Department of Mathematics: Analysis, Logic and Discrete Mathematics
	\endgraf
	Ghent University, Krijgslaan 281, Building S8, B 9000 Ghent
	\endgraf
	Belgium
	\endgraf
	and
	\endgraf
    School of Mathematical Sciences
    \endgraf
    Queen Mary University of London
    \endgraf
    United Kingdom
    \endgraf
	{\it E-mail address} {\rm michael.ruzhansky@ugent.be}}


\newtheoremstyle{theorem}%name
{10pt}          % space above
{10pt}  % space below
{\sl}  % bofy font
{\parindent}     % ident - empty=no indent,  \parindent= paragraph indent
{\bf}  % thm head font
{. }    % punctuation after thm head
{ }    % space after thm head: `` ``=normal \newline=linebreak
{}     % thm head specification
\theoremstyle{theorem}

\numberwithin{equation}{section}
\theoremstyle{plain} 
\newtheorem{thm}{Theorem}[section]
\newtheorem{prop}[thm]{Proposition}
\newtheorem{cor}[thm]{Corollary}
\newtheorem{lem}[thm]{Lemma}

\theoremstyle{definition}
\newtheorem{defn}[thm]{Definition}
\newtheorem{rem}[thm]{Remark}

\newtheoremstyle{defi}%name
{10pt}          % space above
{10pt}  % space below
{\rm}  % bofy font
{\parindent}     % ident - empty=no indent,  \parindent= paragraph indent
{\bf}  % thm head font
{. }    % punctuation after thm head
{ }    % space after thm head: `` ``=normal \newline=linebreak
{}     % thm head specification
\theoremstyle{defi}

\begin{document}
 	\begin{abstract}
Given a smooth, closed Riemannian manifold $(M,g)$ equipped with a linear connection $\nabla$ (not necessarily metric), we develop the holomorphic functional calculus for operators belonging to the global pseudo-differential classes $\Psi_{\rho, \delta}^m\left(\Omega^\kappa, \nabla, \tau\right)$ introduced by Safarov. As a consequence of our main result, we establish a Szeg\"o type-theorem, derive asymptotic expansion of the heat kernel trace, and calculate some associated spectral $\zeta$-functions.
\end{abstract}
	\maketitle
  	\tableofcontents

\section{Introduction}
The purpose of this article is to establish the holomorphic functional calculus for the Safarov pseudo-differential classes on manifolds, extending our previous work \cite{GomezCobosRuzhansky} and further enriching Safarov's theory.

To contextualize the problem, the term {\it functional calculus} generally refers to the process of defining an operator $f(A)$ given an operator $A$ acting on Banach spaces and a function $f$ defined on the real or complex numbers. This process depends on both the properties of $A$ (such as its spectrum and resolvents) and those of $f$ (such as boundedness and continuity). As a result, various types of functional calculi exist, including Borel, $H^\infty$, and holomorphic functional calculus, among others (see, e.g., \cite{Haase, KadisonRingrose}). In this article, we focus on the holomorphic case.

The study of holomorphic functional calculus on manifolds traces back to Seeley’s seminal work \cite{Seeley}, in which he developed a functional calculus for classical pseudo-differential operators. His motivation was to examine the closure of the algebra of such operators and to determine when $f(A)$ remains a classical pseudo-differential operator. Seeley’s results have since found numerous applications in analysis and geometric analysis, including the construction of solutions to certain partial differential equations, the study of $\zeta$-functions associated with operators, and the computation of invariants such as indices and Wodzicki residues.

Over time, many mathematicians have explored this problem in different settings, adapting both the underlying space and the pseudo-differential classes. To mention some contributions we include the works of Widom \cite{w2}, Schrohe \cite{Schrohe}, Vassout \cite{Vass}, Ruzhansky and Wirth \cite{RuzhanskyWirth}, among others. Moreover, after Seeley's work, further developments focused on a more refined treatment of pseudo-differential operators that depend on a parameter. This led to the construction of specific operator classes, rather than merely working with the resolvent operator. An early step in this direction was taken by Shubin \cite{shu}, who introduced what are now known as Shubin classes, covering the case of differential operators. Later, Grubb and Seeley \cite{GrubbSeeley} developed a more general framework that extended to the full class of pseudo-differential operators.

However, given the applications we have in mind, it suffices for us to work directly with the resolvents of operators belonging to Safarov's classes $\Psi_{\rho, \delta}^{m}(\Omega^{\kappa}, \nabla, \tau)$ (as defined in Section \ref{Section 2}). The study of these properties within Safarov's calculus is of particular interest because the classes $\Psi_{\rho, \delta}^{m}(\Omega^{\kappa}, \nabla, \tau)$ extend the classical local H\"ormander classes $\Psi_{\rho, \delta, loc}^m(M)$ on manifolds in various ways. In the H\"ormander setting, the values of $\rho$ and $\delta$ are constrained by the requirement of coordinate invariance, imposing the condition $\rho>1/2$. In contrast, Safarov's framework allows greater flexibility, permitting $\rho>1/3$ and even $\rho>0$. However, this increased freedom comes at the cost of additional geometric conditions. The reason for these new constraints lies in the intrinsic nature of Safarov's classes: they are defined via a linear connection $\nabla$, meaning that restrictions on $\nabla$ directly influence the range of permissible values for $\rho$. Namely, there are three cases: 
\begin{enumerate}
    \item no condition on $\nabla$ and $\rho>\frac{1}{2}$, here the classes $\Psi_{\rho, \delta}^{m}(\Omega^{\kappa}, \nabla, \tau)$ coincide with H\"ormander ones;
    \item the connection $\nabla$ is {\it symmetric} (zero torsion) and $\rho>\frac{1}{3}$, here the classes depend on $\nabla$;
    \item the connection $\nabla$ is {\it flat} (zero torsion and curvature) and no restriction on $\rho>0$, here the classes depend on $\nabla$.
\end{enumerate}
In his original work, Safarov \cite{Safarov} addressed a range of questions, including symbolic calculus, $L^2$- boundedness, ellipticity, etc., and later on dealt with approximate spectral projections \cite{mcK}. Inspired by this approach, Shargorodsky \cite{Shargorodsky} adapted the setting to anisotropic differential operators on manifolds, where he further investigated topics such as $L^p$-boundedness, Fredholm properties, complex powers, and related aspects. 

It is worth noting that this is not the only approach to defining symbol classes on manifolds via connections. Alternative constructions can be found, for example, in the works of Widom \cite{w2, w1}, Fulling–Kennedy \cite{fk}, and Sharafutdinov \cite{sha, sha2}. Furthermore, the idea of employing global symbols on the cotangent bundle has also been developed in other contexts, see, for instance, \cite{Getzler, Voronov2, Voronov1}. 

Now, we introduce our main result. Given $A\in \Psi_{\rho, \delta}^{m}(\Omega^{\kappa}, \nabla, \tau)$ and a suitably regular function $f$, our goal is to define $f(A)$ using a Dunford–Riesz integral:
\begin{equation}\label{intIntro}
    f(A) = \frac{1}{2\pi i}\int_{\Gamma} f(\lambda) (A-\lambda I)^{-1}\, d\lambda,
\end{equation}
where $\Gamma$ is a carefully chosen contour. To ensure that the integral is well-defined, we must analyze the properties of the resolvent operators $(A-\lambda I)^{-1}$. To this end, we introduce the concept of {\it parameter-ellipticity} (see Definition \ref{paraElliDef}), which allows us to derive symbolic estimates and the construction of a parameter-parametrix which approximates the resolvent operator. These, in turn, lead to norm estimates that exhibit decay as $\lambda$ grows (see Theorem \ref{EstimadoFinalResolvente} and Corollary \ref{normResolvent}). Therefore, having these tools at hand, we were able to prove the following result (for precise definitions see Sections \ref{Section 2} and \ref{Section 3}). 
\begin{thm}\label{thmFunctionalCalIntr}
 Let $(M,g)$ be a closed Riemannian manifold and let $\nabla$ be a linear connection (not necessarily metric). Let $0\leq \delta<\rho \leq 1$, $m\geq 0$, $A\in \Psi_{\rho, \delta}^{m}(\Omega^{\kappa}, \nabla, \tau)$, let $\Lambda\subset \bbC$ be a sector and let $f$ be a holomorphic function on $\bbC\setminus\Lambda$. Suppose that at least one of the following conditions is fulfilled:
\begin{enumerate}
    \item $\rho>\frac{1}{2}$;
    \item the connection $\nabla$ is symmetric and $\rho>\frac{1}{3}$;
    \item the connection $\nabla$ is flat.
\end{enumerate}
Suppose that $\sigma_A$, the symbol of $A$, is parameter-elliptic with respect to $\Lambda$. 
\begin{itemize}
    \item If $m>0$, suppose that $f$ satisfies 
 \begin{equation*}
      |f(\lambda)|\leq C |\lambda|^s
 \end{equation*}
 uniformly for some $s\in \bbR$. Then the operator $f(A)$, given by formula \eqref{intIntro}, is well-defined and belongs to the class $\Psi_{\rho, \delta}^{ms}(\Omega^{\kappa}, \nabla, \tau)$.
 \item If $m=0$, then the operator $f(A)$, given by formula \eqref{intIntro}, is well-defined and belongs to the class $\Psi_{\rho, \delta}^{0}(\Omega^{\kappa}, \nabla, \tau)$.
\end{itemize}
Moreover, in both cases, the symbol $\sigma_{f(A)}$ satisfies the asymptotic expansion 
   \[
  \sigma_{f(A)}(y,\eta)\sim f(a(y,\eta)) + \sum_{k=1}^\infty \sum_{\mathfrak{I}_k}(-1)^{1+|\mathfrak{I}_{k,L}|}  \frac{f^{(k+|\mathfrak{I}_{k,L}|)}(a(y,\eta))}{(k+|\mathfrak{I}_{k,L}|)!}\mathfrak{r}_{\mathfrak{I}_k}(y,\eta)
   \]  
   as $\langle\eta\rangle_{y}\to\infty$. Here, $\mathfrak{r}_{\mathfrak{I}_k}$ is a function related to $\sigma_A$ and its derivatives, for precise definition and notation see Remark \ref{remExpAsyPara}. 
\end{thm}
We would like to emphasize that our result extends \cite[Theorem 11.2]{Safarov}, in which Safarov developed a functional calculus for the Laplace–Beltrami operator with an added potential, i.e., operators of the form $\mathcal{L}+\mathfrak{v}$, and for a specific class of functions satisfying certain polynomial growth or decay conditions.

Finally, as applications of our main theorem, we construct three fundamental operators: the complex powers $A^z$, the exponential (or heat) operator $e^{-A}$, and, for zero-order (more generally negative order) operators, the logarithm $\log A$. Moreover, we apply these operators to establish several important results; namely the calculation of the traces of the aforementioned operators. Specifically, we prove a Szeg\"o type-theorem in the sense of Widom \cite{w2} (Theorem \ref{LimitThe}), derive heat kernel trace asymptotics (Theorem \ref{heatKernelExpansion}), and obtain formulas for the spectral $\zeta$-functions associated with $A$ (Theorem \ref{zetaFun}).

\section{Preliminaries}\label{Section 2}
In this section we recall the Safarov pseudo-differential calculus and the relevant results that we will need in Sections \ref{Section 3}, \ref{Section 4} and \ref{Section 5}. Also we construct a trace for this algebra of pseudo-differential operators in a standard way, which will be used in the applications part of the article. Over the whole article we will be using implicitly the Einstein notation.   
\subsection{Some geometric generalities}
First, let us fix the geometric objects and notation we are going to use over the whole document, for more details see, e.g. \cite{berl,kob}. Let $M$ be a $n$-dimensional smooth manifold. We will need densities in order to integrate and define the oscillatory integrals we are going to study, and connections to define the symbol classes and phase functions of such oscillatory integrals. We start with the former. We fix $|dx|$ to be a section of the $1$-density line bundle $\Omega^1\to M$, for a local coordinate system $\{x^k\}$, this means that it is written as
\[
|dx| = \omega(x)\, |dx^1\wedge\cdots \wedge dx^n|
\]
with $\omega$ a positive $C^\infty(M)$ function. If $M$ is a Riemannian manifold, then $\omega(x)= g(x) = \sqrt{|\det(g_{ij}(x))|}$ will do the trick, where $\{g_{ij}\}$ is the Riemannian metric. This section allows us to define integration of functions ($0$-densities), but even more for any $\kappa\in\bbR$ we can define a natural pairing between $\Omega^\kappa$ and $\Omega^{1-\kappa}$ given by integration $\int$ as follows:
\begin{align*}
    \Omega^\kappa\times \Omega^{1-\kappa} &\xrightarrow{( \cdot, \cdot):=\int}\bbC\\
    (\phi, \psi)\hspace{0.5cm} &\xmapsto{\hspace{0.9cm}} \int_M \phi(x) \psi(x) \,|dx|. 
\end{align*}
Moreover, $|dx|$ also induces a measure on the fibers of the cotangent bundle $T^*M\to M$. That is, for a local coordinate system $\{(x^k,\zeta_k)\}$ of $T^*M$, we define a measure on $T_x^*M$ given by 
\begin{equation}\label{fiberForm}
    d\zeta:= d\zeta_x = \omega(x)^{-1}\, d\zeta_1\wedge\cdots \wedge d\zeta_n,
\end{equation}
so that 
\begin{equation}\label{cotangentBundleForm}
    \Omega = d\zeta_x \wedge |dx| =  d\zeta_1\wedge\cdots \wedge d\zeta_n\wedge |dx^1\wedge\cdots \wedge dx^n|
\end{equation}
is the canonical $2n$-form on the cotangent bundle $T^*M$.  

Now, let $\nabla$ be a connection on the cotangent bundle $T^*M$. Notice that since $M$ is not assumed to be Riemannian then we are not assuming that the connection is metric as well (even if it would be Riemannian, we do not need metricity). Associated to a connection there is always given a lot of local geometrical information like geodesics, torsion, curvature, etc. so we will shortly list the objects that we will need.
\begin{itemize}
\item Normal coordinates systems, for them we will use the abbreviation n.c.s. 
\item We denote by $\Gamma_{jk}^i$ the Christoffel symbols associated to $\nabla$,  $T=\{T_{jk}^i\}$ the torsion tensor, $R=\{R_{jkl}^i\}$ the curvature tensor and $\{R_{kl}\}=\{\sum_i R_{kil}^i\}$ the Ricci tensor. If $T\equiv 0$ we say that $\nabla$ is symmetric, and if in addition $R\equiv 0$ we say that $\nabla$ is flat.
\item Given a neighborhood $U_x$ of $x$, we denote by $\gamma_{y,x}(t)$ the shortest geodesic joining $x$ and $y\in U_x$. We will use the notation $z_t= \gamma_{y,x}(t)$.
\item We denote by $\Phi_{y,x}:T_x^*M\to T_y^*M$ the parallel transport along $\gamma_{y,x}(t)$ and we define $\Upsilon_{y,x} := \Upsilon_{x}(y)=|\operatorname{det} \Phi_{y,x}|$, the latter function is a $1$-density in $x$ and a $(-1)$-density in $y$. 
    \end{itemize}
Choosing a connection means that we are choosing a horizontal space $HT^*M$ such that $TT^*M\cong VT^*M\oplus HT^*M$, where $VT^*M$ is the natural vertical part which is generated by vector fields $\{\partial_{\zeta_j}\}\subset \mathfrak{X}(T^*M)$. On the other hand, the connection $\nabla$ will give us the generators of the horizontal part by taking the horizontal lifts of the vector fields $\partial_{x^k}\in \mathfrak{X}(M)$, specifically given a vector field $X=\sum X^{k}(x) \partial_{x^{k}}$ on $M$, its horizontal lift is defined as 
\begin{equation}\label{horizontalLift}
 \nabla_{X}=\sum_{k} X^{k}(x) \partial_{x^{k}}+\sum_{i, j, k} \Gamma_{k j}^{i}(x) X^{k}(x) \zeta_{i}\partial_{\zeta_{j}}.    
\end{equation}
In particular, if $X=\partial_{y^{k}}$ we will denote its horizontal lift by $\nabla_k$, these are the derivatives that will appear in the definition of the symbols classes and they will take over the place of spacial derivatives. 
\subsection{Safarov calculus}
Here we introduce the Safarov pseudo-differential calculus following \cite{Safarov} and state some of the key results for our investigation, for related results and examples we also refer to \cite{GomezCobosRuzhansky}.

Locally the operators are going to be represented by a kernel given by an oscillatory integral, therefore the main ingredients are going to be symbol classes and phase functions. Both of them will be related to the connection $\nabla$ and the local geometry associated to it. We start with the symbols. 
\begin{defn}\label{defSym}
     Let $0\leq \delta<\rho \leq 1$, the space $S_{\rho, \delta}^m(\nabla)$ denotes the class of functions $a\in C^\infty(T^*M)$ such that in any coordinates $y$, for all $\alpha$ and $i_1, \ldots, i_q$, 
$$
\left|\partial_{\eta}^{\alpha} \nabla_{i_{1}} \ldots \nabla_{i_{q}} a(y, \eta)\right| \leq \mathrm{const}_{K, \alpha, i_{1}, \ldots, i_{q}}\langle\eta\rangle_{y}^{m+\delta q-\rho|\alpha|},
$$
where $y$ runs over a compact set $K\subset M$, $\langle\eta\rangle_{y}:= (1 + w^2(y, \eta))^{1/2}$ where $w\in C^\infty(T^*M\setminus 0)$ is a positive function homogeneous in $\eta$ of degree 1, and $\nabla_{i_k}$ is as in \eqref{horizontalLift}. Also amplitudes are defined as satisfying 
$$
\left|\partial_{z}^{\beta} \partial_{\eta}^{\alpha} \nabla_{i_{1}} \ldots \nabla_{i_{q}} a(z ; y, \eta)\right| \leq \mathrm{const}_{K, \alpha, \beta, i_{1}, \ldots, i_{q}}\langle\eta\rangle_{y}^{m+\delta|\beta|+\delta q-\rho|\alpha|}, \text{ for all } y,z\in K.
$$
\end{defn}
As in the classical H\"ormander classes, the classes $S_{\rho, \delta}^m(\nabla)$ satisfy the following properties:
\begin{prop}
The following holds.
     \begin{itemize}
            \item If $a\in S_{\rho, \delta}^{m_1}(\nabla)$ and $b\in S_{\rho, \delta}^{m_2}(\nabla)$, set $m=\operatorname{max}(m_1, m_2)$, then 
            $$ab\in S_{\rho, \delta}^{m_1+m_2}(\nabla) \text{ and } a+b\in S_{\rho, \delta}^{m}(\nabla).$$
            \item Moreover,
            $\partial_\eta^\alpha a \in S_{\rho, \delta}^{m-\rho|\alpha|}(\nabla)$, $\nabla_{v_1}\cdots\nabla_{v_q}a\in S_{\rho, \delta}^{m+\delta q}(\nabla)$
            for all $a\in S_{\rho, \delta}^{m}(\nabla)$ and any $v_1, \ldots, v_q\in \mathfrak{X}(M).$
            \end{itemize}
\end{prop}
\begin{rem}
    Notice that taking horizontal and vertical derivatives of a function naturally produces tensorial objects. For instance, if $a\in C^\infty(T^*M)$, then the collection $\{\partial_\eta^\alpha a\}_{|\alpha|=p}$ forms a $(p,0)$-tensor, while $\{\nabla_{i_{1}} \ldots \nabla_{i_{q}} a\}_q$ defines a $(0,q)$-tensor. For the horizontal derivatives, it is often convenient to work with the following equivalent expression:
\[
   \nabla_x^\alpha a(x,\eta) = \frac{d^\alpha}{dy^\alpha} a(y, \Phi_{y,x}\eta)\bigg|_{y=x},
\]
which, as shown in \cite[Corollary 2.5]{Safarov}, coincides with the iterated covariant derivative and will frequently appear in the formulas that follow.
\end{rem}
We continue with the definition of the phase function. 
\begin{defn}
    Let $V$ be a sufficiently small neighborhood of the diagonal $\Delta$ in $M\times M$. We introduce the phase functions 
$$\varphi_\tau(x,\zeta,y)= -\langle \dot{\gamma}_{y,x}(\tau), \zeta\rangle, \text{ where }(x,y)\in V,\quad \tau\in [0,1],\quad  \zeta\in T^*_{z_\tau}M.$$
\end{defn}
We point out that actually we have introduced a family of phase functions depending on the parameter $\tau$, this on the level of pseudo-differential operators will correspond to the different quantizations one can have. Now we are in position to define Safarov pseudo-differential operators. 
\begin{defn}\label{dfnPseudo}
    Let $0\leq \delta<\rho \leq 1$. Let $A: \Gamma_c(\Omega^\kappa)\to \Gamma(\Omega^\kappa)$ be a linear operator with Schwartz kernel $\mathscr{A}(x, y)$, i.e, $\langle Au, v\rangle = \langle\mathscr{A} , u\otimes v\rangle$. We say that $A$ is pseudo-differential if
      \begin{enumerate}
          \item $\mathscr{A}(x, y)$ is smooth in $(M\times M)\setminus \Delta$.
          \item On a neighborhood $V$ of $\Delta$ the Schwartz kernel is represented by an oscillatory integral of the form
          \begin{equation}\label{oscInt}
              \mathscr{A}(x, y)=\frac{1}{(2\pi)^n}p_{\kappa, \tau} \int_{T^*_{z_\tau}M} e^{i \varphi_{\tau}(x, \zeta, y)} a\left(z_{\tau}, \zeta\right)\,d \zeta,\text { for }(x, y) \in V,
          \end{equation}
          where $a\in S_{\rho, \delta}^m(\nabla)$, $p_{\kappa, \tau} = p_{\kappa, \tau}(x,y) = \Upsilon_{y, z_\tau}^{1-\kappa}\Upsilon_{z_\tau, x}^{-\kappa}$ and $d\zeta$ is as in \eqref{fiberForm}. 
      \end{enumerate}
We denote by $\Psi_{\rho, \delta}^m\left(\Omega^\kappa, \nabla, \tau\right)$ these classes of pseudo-differential operators. Moreover, we say that a pseudo-differential operator $A$ is properly supported if both projections $\pi_1,\pi_2: \operatorname{supp}\mathscr{A}\to M$ are proper maps. We denote by $\Psi^{-\infty}$ the class of operators with smooth kernels and we refer to them as \textit{smoothing operators}. 
\end{defn}
Let us elaborate on this definition for the sake of clarity. Explicitly, the previous definition tells us how pseudo-differential operators can be expressed in concrete terms. Let $\chi$ be a smooth cutoff function equal to $1$ on an open set $V$ and supported within a slightly larger open set. Using this, we define the smooth function
     \begin{equation}\label{kernelDef}
         K(x,y):= (1-\chi(x,y)) \mathscr{A}(x, y) \in C^\infty(M\times M),
     \end{equation}
     so that for any $u\in \Gamma_c(\Omega^\kappa)$ we can write the operator $A$ acting on $u$ as follows:
     \begin{align}\begin{split}
         Au(x) &= A_{loc}u(x) + A_{glo}u(x)\\
         &=\int_M \mathscr{A}(x, y) u(y) \chi(x,y) \,|dy| + \int_M K(x,y)u(y) \,|dy|\\
         &=\frac{1}{(2\pi)^n}\int_M \int_{T_{z_\tau}^*M}  p_{\kappa, \tau}e^{i \varphi_{\tau}(x, \zeta, y)} a\left(z_{\tau}, \zeta\right) \chi(x,y) u(y)d \zeta \,|dy|+ \int_M K(x,y)u(y) \,|dy|.
     \end{split} \label{kernels}     
     \end{align}

Now we briefly discuss how the classes $\Psi_{\rho, \delta}^m\left(\Omega^\kappa, \nabla, \tau\right)$ depend on the parameters $\nabla$ and $\tau$. On the one hand, as shown in \cite[Proposition 4.4]{Safarov}, whenever $0\leq \delta<\rho\leq 1$, the classes are independent of the quantization parameter $\tau$. This allows us to move freely between different quantizations, and for simplicity, we will denote these classes as $\Psi_{\rho, \delta}^m\left(\Omega^\kappa, \nabla\right)$.

On the other hand, if $\rho>1/2$, the classes become independent of the choice of connection $\nabla$, and one recovers the classical local H\"ormander classes. However, when $\rho\leq 1/2$, the dependence on $\nabla$ becomes essential. Even in this case, though, we will be able to choose suitable connections to derive meaningful and useful results, for instance, the Levi-Civita connection. 

In his paper, Safarov established numerous important properties of these pseudo-differential operators. However, for the purposes of our construction, we will restrict ourselves to recalling only those properties that are directly relevant to our work. First, products of operators. 
\begin{thm}{}
\label{producto}
Let $0\leq \delta<\rho \leq 1$. Let $A \in \Psi_{\rho, \delta}^{m_{1}}\left(\Omega^{\kappa}, \nabla\right), B \in \Psi_{\rho, \delta}^{m_{2}}\left(\Omega^{\kappa}, \nabla\right)$, and let at least one of these pseudo-differential operators be properly supported. Assume that at least one of the following conditions is fulfilled:
\begin{enumerate}
    \item $\rho>\frac{1}{2}$;
    \item the connection $\nabla$ is symmetric and $\rho>\frac{1}{3}$;
    \item the connection $\nabla$ is flat.
\end{enumerate}
Then $A B \in \Psi_{\rho, \delta}^{m_{1}+m_{2}}\left(\Omega^{\kappa}, \nabla\right)$ and
\begin{equation}
\label{asyExp}
    \sigma_{A B}(x, \xi) \sim \sum_{\alpha, \beta, \gamma} \frac{1}{\alpha !} \frac{1}{\beta !} \frac{1}{\gamma !} P_{\beta, \gamma}^{(\kappa)}(x, \xi) D_{\xi}^{\alpha+\beta} \sigma_{A}(x, \xi) D_{\xi}^{\gamma} \nabla_{x}^{\alpha} \sigma_{B}(x, \xi)
\end{equation}
as $\langle\xi\rangle_{x} \rightarrow \infty$. In particular, we have that 
\[
\sigma_{A B} - \sigma_A\sigma_B \in S_{\rho, \delta}^{m_1+m_2-r}(\nabla) 
\]
where 
\begin{equation}\label{rGenRhoMinusDelta}
    r=\begin{cases}
           \min\{\rho-\delta, 2\rho-1\} \, &\text{ under condition } (1)\\
           \min\{\rho-\delta, \frac{1}{2}(3\rho-1)\} \, &\text{ under condition } (2)\\
           \rho-\delta \, &\text{ in other cases.} 
       \end{cases}
\end{equation}

\end{thm} 
\begin{rem}
    The functions $P_{\beta, \gamma}^{(\kappa)}(x, \xi)$ are providing extra decay because they are polynomial in $\xi$ with coefficients depending on the curvature, torsion and its covariant derivatives, thus allowing us to consider smaller values of $\rho$. They are defined as follows: 
    \[
    P_{\beta, \gamma}^{(\kappa)}(x, \xi) = \left((\partial_y+\partial_z)^\beta \partial_y^\gamma \sum_{|\beta'|\leq |\beta|} \frac{1}{\beta'!}D_\xi^{\beta'}\partial_y^{\beta'}(e^{i\psi}\Upsilon_{\kappa})\right)\bigg|_{y=z=x}, 
    \]
    where 
    \begin{align*}
        &\Upsilon_{\kappa}(x,y,z)=\Upsilon_{y,z}^{1-\kappa}\Upsilon_{z,x}^{2-\kappa}\Upsilon_{x,y}^{1-\kappa}, \\
        \psi(x,\xi;y,&z)= \langle\dot{\gamma}_{y,x},\xi\rangle-\langle\dot{\gamma}_{z,x},\xi\rangle-\langle\dot{\gamma}_{y,z},\Phi_{z,x}\xi\rangle,
    \end{align*}
    in normal coordinates. 
\end{rem}
Now we give the $L^2$-boundedness result, which is extended to $L^p$-spaces by \cite[Theorem 1.8]{GomezCobosRuzhansky}.
\begin{thm}
\label{L2bound}
Let $0\leq \delta<\rho \leq 1$. Let at least one of Conditions (1)-(3) of Theorem \ref{producto} be fulfilled. Then $A\in \Psi_{\rho, \delta}^{m}\left(\Omega^{\kappa}, \nabla\right)$ is bounded from $H^s_{\operatorname{comp}}(M, \Omega^\kappa)$ to $H^{s-m}_{\operatorname{loc}}(M, \Omega^\kappa)$ for all $s\in \bbR$.
\end{thm}
\begin{rem}
    If $M$ is compact, then the previous result extends globally, i.e. $A$ would be bounded from $H^s(M, \Omega^\kappa)$ to $H^{s-m}(M, \Omega^\kappa)$. 
\end{rem}
We close this subsection by presenting the ellipticity results for these classes of pseudo-differential operators.
\begin{defn}\label{ellip}
     Let $0\leq \delta<\rho \leq 1$, the space $HS_{\rho, \delta}^{m,m_0}(\nabla)$ denotes the subclass of symbols $a\in S_{\rho, \delta}^{m}(\nabla)$ satisfying the following conditions: for any compact subset $K\subset M$ there exist a positive constant $c_K$ such that 
         \[
        |a(y,\eta)^{-1}| \leq \mathrm{const}_{K} \langle\eta\rangle_{y}^{-m_0}, \, \text{ for all } (y,\eta)\in T^*M \text{ with } y\in K, \langle\eta\rangle_{y}\geq c_K,
         \]
         and 
         \[
         \left|\partial_{\eta}^{\alpha} \nabla_{i_{1}} \ldots \nabla_{i_{q}} a(y, \eta)\right| \leq \mathrm{const}_{K, \alpha, i_{1}, \ldots, i_{q}}\langle\eta\rangle_{y}^{\delta q-\rho|\alpha|}|a(y,\eta)|
         \]
         for all $\alpha, i_1, \ldots, i_q$, and $(y,\eta)\in T^*M \text{ with } y\in K, \langle\eta\rangle_{y}\geq c_K$, where again $\langle\eta\rangle_{y}:= (1 + w^2(y, \eta))^{1/2}$. 
         
         We will say that $a\in S_{\rho, \delta}^{m}(\nabla)$ is \textit{elliptic} if $a\in HS_{\rho, \delta}^{m,m}(\nabla)$. If $m\neq m_0$ the previous definition refers to \textit{hypoelliptic} symbols.
\end{defn}
The notion of ellipticity is useful to find parametrices of operators, i.e., to prove the invertibility of an operator modulo smoothing operators. Specifically, we say that an operator $B$ is a \textit{parametrix} of an operator $A$ if 
\[
AB=BA=I + S,
\]
where $S\in \Psi^{-\infty}$. Safarov proved the existence of parametrixes. 
\begin{comment}
    \begin{lem}
    Let $0\leq \delta<\rho \leq 1$. Let $a\in HS_{\rho, \delta}^{m,m_0}(\nabla)$, then $a^{-1}\in HS_{\rho, \delta}^{-m_0,-m}(\nabla)$. 
\end{lem}
\end{comment}

\begin{thm}\label{Parametrix}
    Let $0\leq \delta<\rho \leq 1$. Let at least one of the conditions (1)-(3) of Theorem \ref{producto} be fulfilled. Then any $A\in H\Psi_{\rho,\delta}^{m,m_0}(\Omega^\kappa, \nabla)$ has a pseudodifferential parametrix $B\in H\Psi_{\rho,\delta}^{-m_0,-m}(\Omega^\kappa, \nabla)$. 
\end{thm}
\subsection{Bessel potentials}\label{besselpotentials}
In this subsection, we briefly recall the definition of Bessel potentials, discuss how they fit into the framework of Safarov's classes, and describe their relationship with Sobolev spaces. For further details, we refer the reader to, for example, \cite{heb, stri}. We denote by $\Psi_{\rho, \delta, loc}^{\lambda}$ the classical local H\"ormander classes on manifolds. 

Let  $(M,g)$ be a complete Riemannian manifold, recall that the Laplace-Beltrami operator associated to the metric $g$ is given by 
\[
\mathcal{L}u(x) = \frac{1}{g(x)}\partial_{x^i}\left(g(x) g^{ij}(x)\partial_{x^j}u(x)\right),
\]
where $g_{ij}$ and $g^{ij}$ are the metric and inverse metric coefficients, respectively, and $g(x):=\sqrt{|\operatorname{det}g_{ij}(x)|}$. Notice that we can define $\mathcal{L}$ on $\kappa$-densities due to our fixed section $|dx|$ as follows: $|dx|^\kappa \mathcal{L}(u|dx|^{-\kappa})$. The Bessel potential is defined as the operator
\[
\mathcal{B} = (1-\mathcal{L})^{1/2}. 
\]
Now we list some of the properties of the operators $\mathcal{L}$ and $\mathcal{B}$. 
\begin{itemize}
    \item $\mathcal{L}\in  \Psi_{1, 0}^{2}\left(\Omega^{\kappa}, \nabla\right)$, and for the Levi-Civita connection $\nabla_{LC}$ its symbol is given by
    \[
    a_{\mathcal{L},\tau}(x,\xi)=-|\xi|_x^2+\frac{1}{3}S(x),
    \]
    where $|\xi|_x^2=g^{ij}(x)\xi_i\xi_j$ and $S$ is the scalar curvature. For another connection we will have to add some more terms related to Christoffel symbols. 
    \item Classically $\mathcal{B}^\ell\in \Psi_{1, 0, loc}^{\ell}\left(\Omega^{\kappa}\right)$ for any $\ell\in \bbR$, therefore $\mathcal{B}^\ell\in \Psi_{1, 0}^{\ell}\left(\Omega^{\kappa},\nabla\right)$ for any connection $\nabla$  and any $\ell\in \bbR$ since Safarov classes coincide with the classical ones when $\rho>1/2$. Hence, the obvious inclusion $\Psi_{1, 0}^{m}\left(\Omega^{\kappa}, \nabla\right)\subseteq \Psi_{\rho, \delta}^{m}\left(\Omega^{\kappa}, \nabla\right)$ for $0\leq \delta<\rho\leq 1$ gives us that $\mathcal{B}^\ell\in \Psi_{\rho, \delta}^{\ell}\left(\Omega^{\kappa}, \nabla\right)$ for all $\nabla$ and $\ell \in\bbR$. 
    \item For every $\ell\in\bbR$ the operator $\mathcal{B}^\ell$ is an isomorphism from the Sobolev space $H_{loc}^\ell(M,\Omega^\kappa)$ to $L_{loc}^2(M, \Omega^\kappa)$, which helps us at the moment of calculating these norms.  
\end{itemize}

\subsection{A trace on \texorpdfstring{$\Psi^{<-n}$}{L}}\label{traceSubsection} 
Finally, we construct a trace on the algebra of Safarov pseudo-differential operators. We will use it for our applications in Section \ref{Section 5}. In this part we refer the reader to \cite{Scott} for more details and generalities about traces and determinants on differential graded algebras. In this subsection we assume that $M$ is compact. 

The construction of traces, and more generally quasi-traces, on algebras of pseudo-differential operators depends on the knowledge of the singularity structure of their Schwartz kernel around its singular support. That is why, if we look at the formula \eqref{kernels}, we will just use the kernel $\mathscr{A}$ and not $K$. The trace will be defined for pseudo-differential operators of order $<-n$, where $n$ is the dimension of $M$, because for them the regularization process one applies in order to makes sense of the oscillatory integral \eqref{oscInt} is not needed. 

Let $-\mathfrak{n}<-n$. Let $A\in \Psi_{\rho, \delta}^{-\mathfrak{n}}\left(\Omega^{\kappa}, \nabla\right)$, then by Theorem \ref{L2bound}, $A$ is bounded from $H^{-\mathfrak{n}/2}$ to $H^{\mathfrak{n}/2}$, so that the function 
\begin{align*}
    H^{\mathfrak{n}/2} \times M &\to \bbC\\
    (\varphi, x)\hspace{0.2cm}&\mapsto \varphi(x)
\end{align*}
is continuous. Hence the kernel $\mathscr{A}$ is a continuous function
\[
x \mapsto \mathscr{A}(x, \cdot)
\]
from $M$ to $H^{\mathfrak{n}/2}$ such that for $f\in H^{-\mathfrak{n}/2}$ 
\[
Au(x)= \int_M \mathscr{A}(x, y) f(y) \,|dy| + \int_M K(x,y)f(y)\, |dy|.
\]
Furthermore, the integral 
\[
\int_M \mathscr{A}(x, x)\, |dx| 
\]
is well defined. 
\begin{defn}[Construction of trace]
Let $\Psi^{<-n}$ denote the union of all the classes of pseudo-differential operators with order strictly less than $-n$. Then 
\begin{align}\label{trazaDef}
\begin{split}
    \tr: \Psi^{<-n} &\to \bbC\\
    A&\mapsto \int_M \mathscr{A}(x, x)\, |dx|
\end{split}  
\end{align}
is a well-defined trace. We say that such operators are of trace class. 
\end{defn}
\begin{rem}\label{smoothIsTrace}
    Notice that $\Psi^{-\infty}\subset \Psi^{<-n}$, hence all smoothing operators are trace class, and moreover this trace recovers the classical trace on smoothing operators. 
\end{rem}

Later on we will be interested on computing such traces and getting asymptotic expansions for them, so let us provide some simple results in that direction. 

\begin{prop}\label{kernelEnxx}
    Let $0\leq \delta<\rho \leq 1$ and $-\mathfrak{n}<-n$. Let $A\in \Psi_{\rho, \delta}^{-\mathfrak{n}}\left(\Omega^{\kappa}, \nabla\right)$, then 
    \[
    \mathscr{A}(x,x) = \frac{1}{(2\pi)^n}\int_{T_x^*M} a(x,\zeta) \, d\zeta.
    \]
\end{prop}
\begin{proof}
    Immediate from Definition \ref{dfnPseudo} and previous considerations. 
\end{proof}
Hence Proposition \ref{kernelEnxx} and Definition \ref{trazaDef} imply right away the following result:
\begin{thm}\label{traza}
    Let $0\leq \delta<\rho \leq 1$ and $-\mathfrak{n}<-n$. Let $A\in \Psi_{\rho, \delta}^{-\mathfrak{n}}\left(\Omega^{\kappa}, \nabla\right)$, then the trace of $A$ is given by 
    \[
    \tr A = \frac{1}{(2\pi)^n} \int_{T^*M} a(x,
    \zeta) \, \Omega, 
    \]
    where $\Omega$ is as in \eqref{cotangentBundleForm}. 
\end{thm}

\section{Resolvent of pseudo-differential operators}\label{Section 3}
In this section, we analyze key properties of the resolvent operators $(A-\lambda I)^{-1}$, where $A\in \Psi_{\rho, \delta}^{m}\left(\Omega^{\kappa}, \nabla\right)$. This study will enable us to define and investigate the properties of the Dunford–Riesz integral:
\begin{equation*}
    f(A) = \frac{1}{2\pi i}\int_{\Gamma} f(\lambda) (A-\lambda I)^{-1}\, d\lambda. 
\end{equation*}
Specifically, we establish symbol-like estimates for the parameter symbol of the resolvent and subsequently construct a parameter-parametrix that approximates the resolvent operator. Henceforth, given a symbol $a\in S_{\rho, \delta}^m(\nabla)$, we denote by $a(X,D)$ the corresponding pseudo-differential operator in the class $\Psi_{\rho, \delta}^{m}\left(\Omega^{\kappa}, \nabla\right)$.

\subsection{Parameter-ellipticity}\label{parameter-ellipticity}
First of all, such Dunford-Riesz integral can only make sense if we have some control on the resolvent operator. Explicitly, what we will need is the notion of \textit{Agmon angle} or \textit{ray of minimal growth} \cite{Agmon}. 
\begin{defn}\label{agmonAngle}
    Let $H$ be a Hilbert space. A ray $R_\theta=\{\lambda\in \bbC : \operatorname{arg}(\lambda)=\theta\}$ is called a ray of minimal growth for an operator $A:D(A)\subseteq H\to H$, if $R_\theta$ belongs to the resolvent set of $A$ and 
    \begin{equation}\label{ineWeWant}
    \|(A-\lambda I)^{-1}\|_{H\to H}\leq \frac{C}{|\lambda|}
\end{equation}

for some positive constant $C$, and all $\lambda\in R_\theta\setminus\{0\}$. 
\end{defn} 
 For this reason, the concept of parameter-ellipticity is introduced, so that we can perform a rigorous study of the symbol of the resolvent and, actually, what we achieve is to construct a parametrix $A_\lambda^\sharp$ that approximates the resolvent (see Corollary \ref{normResolvent}). We will see that there will be substancial diferences between operators of positive order and operators of order zero, mainly because the latter are bounded due to Theorem \ref{L2bound}. For now, we consider $M$ to be an $n$-dimensional smooth manifold. Additional properties of $M$ will be introduced as needed throughout the section.
\begin{defn}\label{paraElliDef}
    Let $0\leq \delta<\rho \leq 1$ and let $\Lambda\subset \bbC$ be a sector. 
    \begin{itemize}
        \item For $m>0$, we say that $a \in S_{\rho, \delta}^{m}(\nabla)$ is \textit{parameter-elliptic} with respect to $\Lambda$ if for any compact subset $K\subset M$ there exist a positive constant $c_K$ such that 
         \[
        |\left(a(y,\eta)-\lambda \right)^{-1}| \leq \mathrm{const}_{K} (|\lambda|^{1/m} + \langle\eta\rangle_{y})^{-m}, 
        \]
        for all $\lambda\in\Lambda$ and $(y,\eta)\in T^*M \text{ with } y\in K, |\lambda|+\langle\eta\rangle_{y}\geq c_K$. 
         \item For $m=0$, we say that $a \in S_{\rho, \delta}^{m}(\nabla)$ is \textit{parameter-elliptic} with respect to $\Lambda$ if for any compact subset $K\subset M$ there exist a positive constant $c_K$ such that 
         \[
        |\left(a(y,\eta)-\lambda \right)^{-1}| \leq \mathrm{const}_{K} (|\lambda|+1)^{-1}, 
        \]
        for all $\lambda\in\Lambda$ and $(y,\eta)\in T^*M \text{ with } y\in K, |\lambda|\geq c_K$. 
    \end{itemize} 
\end{defn}

As we mention in the introduction, we are not exaclty constructing parametric classes but still from parameter-ellipticity we can obtain some symbol-like estimates. 

\begin{thm}\label{ParaEllEst}
    Let $0\leq \delta<\rho \leq 1$, $m\geq 0$ and let $\Lambda\subset \bbC$ be a sector. Assume that $a\in S_{\rho, \delta}^{m}(\nabla)$ is parameter-elliptic with respect to $\Lambda$.
    \begin{itemize}
        \item If $m>0$, then for any compact subset $K\subset M$ there exist a positive constant $c_K$ such that 
         \[
        |\partial_\lambda^k\partial_{\eta}^{\alpha} \nabla_{i_{1}} \ldots \nabla_{i_{q}}\left(a(y,\eta)-\lambda \right)^{-1}| \leq C_{K, k, \alpha, i_1, \ldots, i_q} (|\lambda|^{1/m} + \langle\eta\rangle_{y})^{-m(k+1)}\langle\eta\rangle_{y}^{\delta q-\rho|\alpha|}, 
        \]
       for all $k, \alpha, i_1, \ldots, i_q$, $\lambda\in\Lambda$ and $(y,\eta)\in T^*M \text{ with } y\in K, |\lambda|+\langle\eta\rangle_{y}\geq c_K$. In particular, for each fixed $\lambda\in\Lambda$ we have that $(a-\lambda)^{-1}\in S_{\rho, \delta}^{-m}(\nabla)$ for large $\eta$. 
       \item If $m=0$, then for any compact subset $K\subset M$ there exist a positive constant $c_K$ such that 
         \[
        |\partial_\lambda^k\partial_{\eta}^{\alpha} \nabla_{i_{1}} \ldots \nabla_{i_{q}}\left(a(y,\eta)-\lambda \right)^{-1}| \leq C_{K, k, \alpha, i_1, \ldots, i_q} (|\lambda| + 1)^{-(k+1)}\langle\eta\rangle_{y}^{\delta q-\rho|\alpha|}, 
        \]
       for all $k, \alpha, i_1, \ldots, i_q$, $\lambda\in\Lambda$ and $(y,\eta)\in T^*M \text{ with } y\in K, |\lambda|\geq c_K$. In particular, for each fixed $\lambda\in\Lambda$ we have that $(a-\lambda)^{-1}\in S_{\rho, \delta}^{0}(\nabla)$. 
    \end{itemize}

\end{thm}
\begin{proof}
First, let us take $m>0$. We choose $c_K$ to be as the constant coming from parameter-ellipticity. We start by verifying the identities for the horizontal derivatives, we will deduce them from the identity  
\begin{equation}\label{identityInverse}
    (a(y,\eta)-\lambda)^{-1}(a(y,\eta)-\lambda)=1. 
\end{equation}
Applying $\nabla_\ell$ to \eqref{identityInverse} results in 
    \begin{align}\label{q=1}
    \begin{split}
        \nabla_\ell\left[(a(y,\eta)-\lambda)^{-1}(a(y,\eta)-\lambda)\right] &= 0 \iff\\
       \left( \nabla_\ell (a(y,\eta)-\lambda)^{-1}\right) (a(y,\eta)-\lambda) + (a(y,\eta)-\lambda)^{-1} (\nabla_\ell a(y,\eta)) &= 0\iff\\
       \nabla_\ell (a(y,\eta)-\lambda)^{-1} = (a(y,\eta)-\lambda)^{-2}& (\nabla_\ell a(y,\eta)).
    \end{split}  
    \end{align}
    Thus using this expression for $\nabla_\ell (a(y,\eta)-\lambda)^{-1}$, parameter-ellipticity and the estimates for $a(y,\eta)$ we get 
    \begin{align*}
        |\nabla_\ell (a(y,\eta)-\lambda)^{-1}| &= |(a(y,\eta)-\lambda)^{-2} (\nabla_\ell a(y,\eta))|\\
        &\leq C_{K,\ell}(|\lambda|^{1/m} + \langle\eta\rangle_{y})^{-2m}\langle\eta\rangle_{y}^{m+\delta}\\
        &\leq C_{K,\ell} (|\lambda|^{1/m} + \langle\eta\rangle_{y})^{-m}\langle\eta\rangle_{y}^{\delta}. 
    \end{align*}
    We obtain the desired result on $x$-derivatives by performing an induction argument, so that for all $i_1, \ldots, i_q$ we have 
    \[
    |\nabla_{i_{1}} \ldots \nabla_{i_{q}} (a(y,\eta)-\lambda)^{-1}|\leq C_{K, i_1, \ldots, i_q} (|\lambda|^{1/m} + \langle\eta\rangle_{y})^{-m}\langle\eta\rangle_{y}^{\delta q}.
    \]
    Now let us consider the case of vertical derivatives. If $q=0$, i.e. there are no horizontal derivatives, we get from \eqref{identityInverse} again by Leibniz rule
    \[
    \partial_{\eta_l} (a(y,\eta)-\lambda)^{-1} = (a(y,\eta)-\lambda)^{-2} (\partial_{\eta_l} a(y,\eta)),
    \]
    and estimating as before  
    \begin{align*}
        |\partial_{\eta_l} (a(y,\eta)-\lambda)^{-1}| &= |(a(y,\eta)-\lambda)^{-2} (\partial_{\eta_l} a(y,\eta))|\\
        &\leq C_{K,l}(|\lambda|^{1/m} + \langle\eta\rangle_{y})^{-2m}\langle\eta\rangle_{y}^{m-\rho}\\
        &\leq C_{K,l} (|\lambda|^{1/m} + \langle\eta\rangle_{y})^{-m}\langle\eta\rangle_{y}^{-\rho}. 
    \end{align*}
    Hence, again by induction we obtain for all $\alpha$
    \[
    |\partial^\alpha_{\eta} (a(y,\eta)-\lambda)^{-1}|\leq C_{K, \alpha} (|\lambda|^{1/m} + \langle\eta\rangle_{y})^{-m}\langle\eta\rangle_{y}^{-\rho |\alpha|}.
    \]
    On the other hand if $q=1$, we differentiate the identity \eqref{q=1} 
    \begin{align*}
        \partial_{\eta_l} \nabla_\ell (a(y,\eta)-\lambda)^{-1} &= \partial_{\eta_l}\left( (a(y,\eta)-\lambda)^{-2} (\nabla_\ell a(y,\eta))\right)\\
        &= 2 (a(y,\eta)-\lambda)^{-1} (\partial_{\eta_l}a(y,\eta)-\lambda)^{-1})(\nabla_\ell a(y,\eta)) \\
        &+ (a(y,\eta)-\lambda)^{-2} (\partial_{\eta_l}\nabla_\ell a(y,\eta)),
    \end{align*}
    so from the previous estimations, parameter-ellipticity and estimates for $a(y,\eta)$ we obtain  
    \begin{align*}
        |\partial_{\eta_l} \nabla_\ell (a(y,\eta)-\lambda)^{-1}|&\leq C_{K,l,\ell} (|\lambda|^{1/m} + \langle\eta\rangle_{y})^{-m}(|\lambda|^{1/m} + \langle\eta\rangle_{y})^{-m}\langle\eta\rangle_{y}^{-\rho}\langle\eta\rangle_{y}^{m+\delta }\\
        &+(|\lambda|^{1/m}+ \langle\eta\rangle_{y})^{-2m}\langle\eta\rangle_{y}^{m+\delta-\rho}\\
        &\leq C_{K,l,\ell} (|\lambda|^{1/m}+\langle\eta\rangle_{y})^{-m}\langle\eta\rangle_{y}^{\delta-\rho}.
    \end{align*}
    Once more from induction it follows that for all $\alpha$
    \[
    |\partial_{\eta}^\alpha \nabla_\ell (a(y,\eta)-\lambda)^{-1}| \leq C_{K,\alpha,\ell}  (|\lambda|^{1/m} + \langle\eta\rangle_{y})^{-m}\langle\eta\rangle_{y}^{\delta-\rho|\alpha|}.
    \]
    At this moment it becomes clear that carrying out a double induction argument we will get for all $\alpha, i_1, \ldots, i_q$  
     \[
    |\partial_\eta^\alpha\nabla_{i_{1}} \ldots \nabla_{i_{q}} (a(y,\eta)-\lambda)^{-1}|\leq C_{K, \alpha, i_1, \ldots, i_q} (|\lambda|^{1/m} + \langle\eta\rangle_{y})^{-m}\langle\eta\rangle_{y}^{\delta q-\rho|\alpha|}.
    \]
    We conclude by studying the case of $\lambda$-derivatives. Notice that for for $|\alpha|=q=0$ 
    \begin{align*}
        \partial_\lambda\left[(a(y,\eta)-\lambda)^{-1}(a(y,\eta)-\lambda)\right] &= 0 \iff\\
       \left( \partial_\lambda (a(y,\eta)-\lambda)^{-1}\right) (a(y,\eta)-\lambda) + (a(y,\eta)-\lambda)^{-1} (\partial_\lambda (a(y,\eta)-\lambda)) &= 0\iff\\
       \partial_\lambda(a(y,\eta)-\lambda)^{-1} = (a(y,\eta)-\lambda)^{-2}&.
    \end{align*}
    Hence 
    \[
    |\partial_\lambda(a(y,\eta)-\lambda)^{-1}| = |(a(y,\eta)-\lambda)^{-2}|\leq C_K (|\lambda|^{1/m} + \langle\eta\rangle_{y})^{-2m}.
    \]
    We observe that using induction several times will end up implying that for all $k, \alpha, i_1, \ldots, i_q$ 
    \[
        |\partial_\lambda^k\partial_{\eta}^{\alpha} \nabla_{i_{1}} \ldots \nabla_{i_{q}}\left(a(y,\eta)-\lambda \right)^{-1}| \leq C_{K, k, \alpha, i_1, \ldots, i_q} (|\lambda|^{1/m} + \langle\eta\rangle_{y})^{-m(k+1)}\langle\eta\rangle_{y}^{\delta q-\rho|\alpha|}, 
        \]
        as we promised. Furthermore, since for a fixed $\lambda$ the inequality 
        \[
        (|\lambda|^{1/m} + \langle\eta\rangle_{y})^{-m}\langle\eta\rangle_{y}^{\delta q-\rho|\alpha|}\leq \langle\eta\rangle_{y}^{-m+\delta q-\rho|\alpha|}
        \]
        holds, we deduce that $(a-\lambda)^{-1}\in S_{\rho, \delta}^{-m}(\nabla)$ for large $\eta$. 
        Finally, we point out that if $m=0$ the proof is completely analogous. 
\end{proof}
Using these estimations we can construct a parameter-parametrix $A_\lambda^\sharp$ for the operators $a(X,D)-\lambda I$. 
\begin{thm}\label{realEstiParameSymbol}
    Let $0\leq \delta<\rho \leq 1$, $m\geq 0$ and let $\Lambda\subset \bbC$ be a sector. Suppose that at least one of the following conditions is fulfilled:
\begin{enumerate}
    \item $\rho>\frac{1}{2}$;
    \item the connection $\nabla$ is symmetric and $\rho>\frac{1}{3}$;
    \item the connection $\nabla$ is flat.
\end{enumerate}
If $a\in S_{\rho, \delta}^{m}(\nabla)$ is parameter-elliptic with respect to $\Lambda$, then $a(x,D)-\lambda I$ has a parametrix $A^\sharp_\lambda$ such that its symbol satisfies 
\[
   a^\sharp(y,\eta,\lambda):=a_{A^\sharp}(y,\eta,\lambda) \sim (a(y,\eta)-\lambda)^{-1} \sum_{k=0}^\infty b_{k}(y,\eta, \lambda) \, \text{ as } \langle\eta\rangle_{y} \to\infty,
   \]
   where for all $\lambda$ each $b_k\in S_{\rho,\delta}^{-kr}(\nabla)$ for some positive $r$. Moreover, 
   \begin{itemize}
       \item if $m>0$, the symbol  $a^\sharp(y,\eta,\lambda)$ satisfies that for any compact subset $K\subset M$ there exist a positive constant $c_K$ such that 
         \[
        |\partial_\lambda^k\partial_{\eta}^{\alpha} \nabla_{i_{1}} \ldots \nabla_{i_{q}} a^\sharp(y,\eta,\lambda)| \leq C_{K, k, \alpha, i_1, \ldots, i_q} (|\lambda|^{1/m} + \langle\eta\rangle_{y})^{-m(k+1)}\langle\eta\rangle_{y}^{\delta q-\rho|\alpha|}, 
        \]
       for all $k, \alpha, i_1, \ldots, i_q$, $\lambda\in\Lambda$ and $(y,\eta)\in T^*M \text{ with } y\in K, |\lambda|+\langle\eta\rangle_{y}\geq c_K$. In particular, for each fixed $\lambda\in\Lambda$ we have that $A^\sharp_\lambda\in \Psi_{\rho, \delta}^{-m}\left(\Omega^{\kappa}, \nabla\right)$;
       \item if $m=0$, the symbol  $a^\sharp(y,\eta,\lambda)$ satisfies that for any compact subset $K\subset M$ there exist a positive constant $c_K$ such that 
         \[
        |\partial_\lambda^k\partial_{\eta}^{\alpha} \nabla_{i_{1}} \ldots \nabla_{i_{q}} a^\sharp(y,\eta,\lambda)| \leq C_{K, k, \alpha, i_1, \ldots, i_q} (|\lambda| + 1)^{-(k+1)}\langle\eta\rangle_{y}^{\delta q-\rho|\alpha|}, 
        \]
       for all $k, \alpha, i_1, \ldots, i_q$, $\lambda\in\Lambda$ and $(y,\eta)\in T^*M \text{ with } y\in K, |\lambda|\geq c_K$. In particular, for each fixed $\lambda\in\Lambda$ we have that $A^\sharp_\lambda\in \Psi_{\rho, \delta}^{0}\left(\Omega^{\kappa}, \nabla\right)$.
   \end{itemize}
\end{thm}
   \begin{proof}
   Let us deal with the case $m>0$, when $m$ is equal to zero the proof is analogous. For the construction of the parametrix we follow the proof of Theorem \ref{Parametrix}. Notice that in Theorem \ref{ParaEllEst} we attained that $(a-\lambda)^{-1}\in S_{\rho, \delta}^{-m}(\nabla)$ for large $\eta$, which morally prevent $(a-\lambda)^{-1}$ of being a true symbol; in order to correct this we multiply it with a nice function. Let $c_K$ be the constant from parameter-ellipticity and let $\chi$ be a smooth function given by
   \[
   \chi(t) = \begin{cases}
       0, &|t|\leq c_K,\\
       1, &|t|> c_K.
   \end{cases}
   \]
   Thus we set $\chi(y,\eta, \lambda) = \chi (\langle\lambda\rangle + \langle\eta\rangle_{y})$, where $\langle\lambda\rangle= \sqrt{1+|\lambda|^2}$, and we notice that $\chi$ vanishes in a neighborhood of the zeros of $a-\lambda$ and it is identically $1$ for large $\eta$. Let us define 
       \[
       b_0(y,\eta,\lambda) = \chi(y,\eta, \lambda) (a(y,\eta)-\lambda)^{-1}.
       \]
       From Theorem \ref{ParaEllEst} it follows that for each $\lambda\in\Lambda$, we have $b_0\in S_{\rho,\delta}^{-m}(\nabla)$.  We denote by $B_0$ the pseudo-differential operator associated to $b_0$. By Theorem \ref{producto} we have that the symbol of the composition operator $B_0 (a(X,D)-\lambda I)$ satisfies 
       \begin{align*}
       \sigma_{b_0 a}(y,\eta,\lambda) &\sim 1 + \mathfrak{r}(y,\eta,\lambda)\\
           &\sim 1 + \sum_{|\alpha|+|\beta|+|\gamma|\neq 0} \frac{1}{\alpha !} \frac{1}{\beta !} \frac{1}{\gamma !} P_{\beta, \gamma}^{(\kappa)}(y, \eta) D_{\eta}^{\alpha+\beta} b_0(y,\eta,\lambda) D_{\eta}^{\gamma} \nabla_{y}^{\alpha} (a(y,\eta)-\lambda),
       \end{align*}
       and we immediately recognize that $\mathfrak{r}$ satisfies the estimates 
       \begin{equation*}
       |\mathfrak{r}(y,\eta,\lambda)| \leq C_{K} (|\lambda|^{1/m} + \langle\eta\rangle_{y})^{-m}\langle\eta\rangle_{y}^{m-r},
       \end{equation*}
       where $r$ is as in \eqref{rGenRhoMinusDelta}. 
       Moreover, using induction and Theorem \ref{ParaEllEst} we further obtain that 
       \begin{equation}\label{estimadoError}
         |\partial_{\eta}^{\alpha} \nabla_{i_{1}} \ldots \nabla_{i_{q}}\mathfrak{r}(y,\eta,\lambda)| \leq C_{K, k, \alpha, i_1, \ldots, i_q} (|\lambda|^{1/m} + \langle\eta\rangle_{y})^{-m}\langle\eta\rangle_{y}^{m-\rho|\alpha|+\delta q-r},  
       \end{equation}
       so that for all fixed $\lambda\in \Lambda$ we have that $\mathfrak{r}\in S_{\rho, \delta}^{-r}(\nabla)$. Thus $B_0 (a(X,D)-\lambda I) = I + \mathfrak{R}$ with $\mathfrak{R} \in \Psi_{\rho, \delta}^{-r}\left(\Omega^{\kappa}, \nabla\right)$. Now let $E\in \Psi_{\rho, \delta}^{0}\left(\Omega^{\kappa}, \nabla\right)$ be the properly-supported pseudo-differential operator such that 
       \[
       E \sim  \sum_{j=0}^\infty (-1)^{j}\,\mathfrak{R}^j, 
       \]
       i.e. 
       \[
       \sigma_E \sim \sum_{j=0}^\infty (-1)^{j}\,\mathfrak{r}^j,
       \]
       then, from it construction, we have that 
       \[
       (EB_0)(a(X,D)-\lambda) = I \text{ mod } \Psi^{-\infty}.
       \]
       We conclude that $A_{\lambda,1}^\sharp:= EB_0$ is a left parametrix of $a(X,D)-\lambda I$. Analogously we can construct a right parametrix $A_{\lambda,2}^\sharp$ such that $A_{\lambda,1}^\sharp = A_{\lambda,2}^\sharp$ $(\text{mod }\Psi^{-\infty})$. Therefore, we obtain that $A_{\lambda}^\sharp:= A_{\lambda,1}^\sharp$ is a parametrix, and by construction its symbol satisfies 
       \begin{equation}\label{asympPara}
              a^\sharp(y,\eta,\lambda):=a_{A^\sharp}(y,\eta,\lambda) \sim (a(y,\eta)-\lambda)^{-1} \sum_{k=0}^\infty b_{k}(y,\eta, \lambda) \, \text{ as } \langle\eta\rangle_{y} \to\infty,
       \end{equation}
   where $b_k=(-1)^k\mathfrak{r}^k$ and for all $\lambda$ each $b_k\in S_{\rho,\delta}^{-kr}(\nabla)$. Finally, we obtain the corresponding estimates for $a^\sharp$ utilizing \eqref{asympPara}, Theorem \ref{ParaEllEst} and the Leibniz rule. Once we have such estimates in our hands, we can assure that $a^{\sharp}$ is a true symbol in the class $S_{\rho, \delta}^{-m}(\nabla)$, so that for a fixed $\lambda$ the parametrix $A_\lambda^\sharp$ is a pseudo-differential operator in the class $\Psi_{\rho, \delta}^{-m}\left(\Omega^{\kappa}, \nabla\right)$, finishing the proof. 
   \end{proof}
   \begin{rem}\label{remExpAsyPara}
        Let us be more precise on the asymptotic expansion \eqref{asympPara}. We come back to the definition of $\mathfrak{r}$ and expand the terms using the Leibniz rule
        \begin{align*}
            \mathfrak{r}(y,\eta,\lambda) =& \sum_{|\alpha|+|\beta|+|\gamma|\neq 0} \frac{1}{\alpha !} \frac{1}{\beta !} \frac{1}{\gamma !} P_{\beta, \gamma}^{(\kappa)}(y, \eta) D_{\eta}^{\alpha+\beta} b_0(y,\eta,\lambda) D_{\eta}^{\gamma} \nabla_{y}^{\alpha} (a(y,\eta)-\lambda)\\
            =& \sum_{|\alpha|+|\beta|+|\gamma|\neq 0} \frac{1}{\alpha !} \frac{1}{\beta !} \frac{1}{\gamma !} P_{\beta, \gamma}^{(\kappa)}(y, \eta) D_{\eta}^{\alpha+\beta} [\chi(y,\eta,\lambda)(a(y,\eta)-\lambda)^{-1}] D_{\eta}^{\gamma} \nabla_{y}^{\alpha} a(y,\eta)\\
            =& \sum_{|\alpha|+|\beta|+|\gamma|\neq 0} \sum_{\iota\leq \alpha+\beta} \frac{1}{\alpha !} \frac{1}{\beta !} \frac{1}{\gamma !}\frac{(\alpha+\beta)!}{\iota!(\alpha+\beta-\iota)!} P_{\beta, \gamma}^{(\kappa)}(y, \eta)  D_{\eta}^{\alpha+\beta-\iota}\chi(y,\eta,\lambda)\\
            & D_{\eta}^{\iota}(a(y,\eta)-\lambda)^{-1} D_{\eta}^{\gamma} \nabla_{y}^{\alpha} a(y,\eta)\\
            =& \sum_{|\alpha|+|\beta|+|\gamma|\neq 0} \sum_{\iota\leq \alpha+\beta} \frac{1}{\alpha !} \frac{1}{\beta !} \frac{1}{\gamma !}\frac{(\alpha+\beta)!}{\iota!(\alpha+\beta-\iota)!}(-1)^\iota \iota! P_{\beta, \gamma}^{(\kappa)}(y, \eta)  D_{\eta}^{\alpha+\beta-\iota}\chi(y,\eta,\lambda)\\
            &(a(y,\eta)-\lambda)^{-1-|\iota|}D_{\eta}^{\iota}a(y,\eta) D_{\eta}^{\gamma} \nabla_{y}^{\alpha} a(y,\eta)\\
            =& \sum_{|\alpha|+|\beta|+|\gamma|\neq 0} \sum_{\iota\leq \alpha+\beta} \mathfrak{r}_{1,\alpha,\beta,\gamma,\iota}(y,\eta,\lambda)(a(y,\eta)-\lambda)^{-1-|\iota|}.
        \end{align*}
        We have defined the function $\mathfrak{r}_{1,\alpha,\beta,\gamma,\iota}$ as a polynomial in the variables $P_{\beta, \gamma}^{(\kappa)}$, $\chi$, $\partial_{\eta_l}\chi$, $\partial_{\eta_l}a$ and $\nabla_l a$. Importantly, it contains no terms involving $(a(y,\eta)-\lambda)^{-1}$, which is significant when deriving asymptotic expansions for the symbols of functions of pseudo-differential operators. Because of this, we aim to derive analogous expressions for higher powers of $\mathfrak{r}$. In particular, we now consider the case of $\mathfrak{r}^2$. Applying the product formula for symbols and the Leibniz rule yields the following:
        \begin{align*}
           \mathfrak{r}^2(y,\eta,\lambda)=& \sum_{\alpha_2,\beta_2,\gamma_2} \frac{1}{\alpha_2 !} \frac{1}{\beta_2 !} \frac{1}{\gamma_2 !} P_{\beta_2, \gamma_2}^{(\kappa)}(y, \eta) D_{\eta}^{\alpha_2+\beta_2} \sum_{|\alpha|+|\beta|+|\gamma|\neq 0} \sum_{\iota\leq \alpha+\beta} \mathfrak{r}_{1,\alpha,\beta,\gamma,\iota}(y,\eta,\lambda)\\
           &(a(y,\eta)-\lambda)^{-1-|\iota|} D_{\eta}^{\gamma_2} \nabla_{y}^{\alpha_2}\sum_{|\alpha'|+|\beta'|+|\gamma'|\neq 0} \sum_{\iota'\leq \alpha'+\beta'} \mathfrak{r}_{1,\alpha',\beta',\gamma',\iota'}(y,\eta,\lambda)\\
           &(a(y,\eta)-\lambda)^{-1-|\iota'|}\\
           =& \sum_{\alpha_2,\beta_2,\gamma_2}\sum_{|\alpha|+|\beta|+|\gamma|\neq 0} \sum_{\iota\leq \alpha+\beta}\sum_{|\alpha'|+|\beta'|+|\gamma'|\neq 0} \sum_{\iota'\leq \alpha'+\beta'} \frac{1}{\alpha_2 !} \frac{1}{\beta_2 !} \frac{1}{\gamma_2 !} P_{\beta_2, \gamma_2}^{(\kappa)}(y, \eta)\\
           &\sum_{\iota_2\leq \alpha_2+\beta_2} \frac{(\alpha_2+\beta_2)!}{\iota_2!(\alpha_2+\beta_2-\iota_2)!}\frac{(-1)^{|\iota_2|}(\iota+\iota_2)!}{\iota!} D_{\eta}^{\alpha_2+\beta_2-\iota_2}  \mathfrak{r}_{1,\alpha,\beta,\gamma,\iota}(y,\eta,\lambda)\\
           &(a(y,\eta)-\lambda)^{-1-|\iota|-|\iota_2|}D_{\eta}^{\iota_2}a(y,\eta) D_{\eta}^{\gamma_2}  \sum_{\mu_2\leq \alpha_2} \frac{\alpha_2!}{\mu_2!(\alpha_2-\mu_2)!}\frac{(-1)^{|\mu_2|}(\iota'+\mu_2)!}{\iota'!}\\
           &\nabla_{y}^{\alpha_2-\mu_2}\mathfrak{r}_{1,\alpha',\beta',\gamma',\iota'}(y,\eta,\lambda)(a(y,\eta)-\lambda)^{-1-|\iota'|-|\mu_2|}\nabla_{y}^{\mu_2}a(y,\eta)\\
           =& \sum_{\alpha_2,\beta_2,\gamma_2}\sum_{|\alpha|+|\beta|+|\gamma|\neq 0} \sum_{\iota\leq \alpha+\beta}\sum_{|\alpha'|+|\beta'|+|\gamma'|\neq 0} \sum_{\iota'\leq \alpha'+\beta'} \frac{1}{\alpha_2 !} \frac{1}{\beta_2 !} \frac{1}{\gamma_2 !} P_{\beta_2, \gamma_2}^{(\kappa)}(y, \eta)\\
           &\sum_{\iota_2\leq \alpha_2+\beta_2} \frac{(\alpha_2+\beta_2)!}{\iota_2!(\alpha_2+\beta_2-\iota_2)!}\frac{(-1)^{|\iota_2|}(\iota+\iota_2)!}{\iota!} D_{\eta}^{\alpha_2+\beta_2-\iota_2}  \mathfrak{r}_{1,\alpha,\beta,\gamma,\iota}(y,\eta,\lambda)\\
           &(a(y,\eta)-\lambda)^{-1-|\iota|-|\iota_2|}D_{\eta}^{\iota_2}a(y,\eta)   \sum_{\mu_2\leq \alpha_2} \frac{\alpha_2!}{\mu_2!(\alpha_2-\mu_2)!}\frac{(-1)^{|\mu_2|}(\iota'+\mu_2)!}{\iota'!}\\
           &\sum_{\nu_2\leq \gamma_2} \frac{\gamma_2!}{\nu_2!(\gamma_2-\nu_2)!}\frac{(-1)^{|\nu_2|}(\iota'+\mu_2+\nu_2)!}{(\iota'+\mu_2)!}D_{\eta}^{\gamma_2-\nu_2}[\nabla_{y}^{\alpha_2-\mu_2}\mathfrak{r}_{1,\alpha',\beta',\gamma',\iota'}(y,\eta,\lambda)\\
           &\nabla_{y}^{\mu_2}a(y,\eta)](a(y,\eta)-\lambda)^{-1-|\iota'|-|\mu_2|-|\nu_2|}D_{\eta}^{\nu_2}a(y,\eta)\\
           =& \sum_{\alpha_2,\beta_2,\gamma_2}\sum_{|\alpha|+|\beta|+|\gamma|\neq 0} \sum_{\iota\leq \alpha+\beta}\sum_{|\alpha'|+|\beta'|+|\gamma'|\neq 0} \sum_{\iota'\leq \alpha'+\beta'}\sum_{\iota_2\leq \alpha_2+\beta_2}\sum_{\mu_2\leq \alpha_2}\sum_{\nu_2\leq \gamma_2} \\
           &\mathfrak{r}_{2,\alpha,\beta,\gamma,\iota, \alpha',\beta',\gamma',\iota',\alpha_2,\beta_2,\gamma_2,\iota_2,\mu_2,\nu_2}(y,\eta,\lambda)(a(y,\eta)-\lambda)^{-2-|\iota|-|\iota'|-|\iota_2|-|\mu_2|-|\nu_2|}.
        \end{align*}
        In the end, we obtained the desired form using the function $\mathfrak{r}_{2,\cdot}$, which is more intricate than $\mathfrak{r}_{1,\cdot}$ and, in fact, depends on it. Additionally, it involves two new indices arising from the application of two extra Leibniz rules.
        For the general case we notice that on each step new coefficients $\alpha_i, \beta_i,\gamma_i,\iota_i,\mu_i,\nu_i$ will appear and no more than those, i.e. we will always have 3 indices from the product formula and 3 more new indices coming from the Leibniz formula. Therefore, we introduce a short notation just to make the formuli more compact. Let $\mathfrak{I}_1$ be the set of the first indices, that is $\mathfrak{I}_1=\{\alpha,\beta,\gamma,\iota\}$, then the next set of indices is defined as $\mathfrak{I}_2=\{\mathfrak{I}_1, \mathfrak{I}_1', \alpha_2,\beta_2,\gamma_2,\iota_2,\mu_2,\nu_2\}$, and recursively we have that the $k$-th set of indices is defined by 
        \[
        \mathfrak{I}_k = \{\mathfrak{I}_{k-1}, \mathfrak{I}_{k-1}', \alpha_k,\beta_k,\gamma_k,\iota_k,\mu_k,\nu_k\}. 
        \]
        Moreover, we recognize that the multi indices appearing as a power on $(a(y,\eta)-\lambda)^{-1}$ are always the ones coming from the Leibniz rule, so we denote that subset as $\mathfrak{I}_{k,L}\subset \mathfrak{I}_k$.
        
        All in all,  with this new notation the powers $\mathfrak{r}^k$, for integers $k\geq1$, can be written as 
        \[
        \mathfrak{r}^k(y,\eta,\lambda) = \sum_{\mathfrak{I}_k} \mathfrak{r}_{\mathfrak{I}_k}(y,\eta,\lambda)(a(y,\eta)-\lambda)^{-k-|\mathfrak{I}_{k,L}|},
        \] 
        where $\mathfrak{r}_{\mathfrak{I}_k}$ is defined recursively and it does not contain any term involving $(a(y,\eta)-\lambda)^{-1}$. Therefore, we can write the expansion \eqref{asympPara} as follows:
        \begin{align*}
            \begin{split}
                a^\sharp(y,\eta,\lambda) &\sim (a(y,\eta)-\lambda)^{-1} \sum_{k=0}^\infty b_{k}(y,\eta, \lambda)\\
                &\sim (a(y,\eta)-\lambda)^{-1} \sum_{k=0}^\infty \sum_{\mathfrak{I}_k}(-1)^k  \mathfrak{r}_{\mathfrak{I}_k}(y,\eta,\lambda)(a(y,\eta)-\lambda)^{-k-|\mathfrak{I}_{k,L}|}\\
                &\sim \sum_{k=0}^\infty \sum_{\mathfrak{I}_k}(-1)^k  \mathfrak{r}_{\mathfrak{I}_k}(y,\eta,\lambda)(a(y,\eta)-\lambda)^{-1-k-|\mathfrak{I}_{k,L}|}\\
                &\sim  (a(y,\eta)-\lambda)^{-1} + \sum_{k=1}^\infty \sum_{\mathfrak{I}_k}(-1)^k  \mathfrak{r}_{\mathfrak{I}_k}(y,\eta,\lambda)(a(y,\eta)-\lambda)^{-1-k-|\mathfrak{I}_{k,L}|}
            \end{split}
        \end{align*}
        as $\langle\eta\rangle_{y} \to\infty$. 
       
       Finally, we observe that the dependence of $\mathfrak{r}_{\mathfrak{I}_k}$ on the variable $\lambda$ comes only through the function $\chi$ and its derivatives. However, since $\chi$ is identically 1 for sufficiently large values of $\langle\eta\rangle_{y}$, this dependence vanishes in the asymptotic expansion. Specifically, we get the following asymptotic expansion for the symbol of the parameter-parametrix
        \begin{equation}\label{asymPara2}
             a^\sharp(y,\eta,\lambda) \sim (a(y,\eta)-\lambda)^{-1} + \sum_{k=1}^\infty \sum_{\mathfrak{I}_k}(-1)^k  \mathfrak{r}_{\mathfrak{I}_k}(y,\eta)(a(y,\eta)-\lambda)^{-1-k-|\mathfrak{I}_{k,L}|}
        \end{equation}
        as $\langle\eta\rangle_{y} \to\infty$.
   \end{rem}
   \begin{rem}\label{extensionToSmallBall}
       In the next section, we will consider operators with discrete spectrum such that zero does not belong to the spectrum. Consequently, there exists a small $\epsilon>0$ for which the ball of radius $\epsilon$ around zero does not intersect the spectrum. Therefore, using parameter-ellipticity of $a\in S_{\rho, \delta}^{m}(\nabla)$ we can extend the previous results to hold in the extended region $\Lambda_\epsilon = \Lambda\cup \{\lambda\in\bbC: |\lambda|<\epsilon\}$. 
   \end{rem}
   \subsection{Norm estimates of resolvent operators}
   
   Here we take $(M,g)$ to be a closed Riemannian manifold with a connection $\nabla$ (not neccesarily metric). 
   
   We recall that our objective is to obtain some control of the resolvent of the type \eqref{ineWeWant}, thus now we will focus on computing some operator norms to achieve that. Specifically, the space $X$ will be the Hilbert space $L^2(M)$ (not just $L_{loc}^2$ because now $M$ is compact) or more generally some Sobolev space $H^s(M)$. Keeping this in mind we now introduce an auxiliary operator $\mathfrak{B}_{\lambda}^k$ that will play the role of Bessel potentials (see Subsection \ref{besselpotentials}) in this parameter operators framework and will help us to produce the estimates we require. 
   \begin{defn}\label{auxOperator}
       Let $k\in\bbR$ and let $\Lambda\subset \bbC$ be a sector. For coordinates $\{(y^k,\eta_k)\}$ of $T^*M$ let us consider the function $\mathfrak{b}_k(y,\eta,\lambda)= (|\lambda|^{1/|k|} + \langle\eta\rangle_{y})^{k}$, which for fixed $\lambda$ belongs to $S_{\rho,\delta}^{k}(\nabla)$, and we define the parameter pseudo-differential operator $\mathfrak{B}_{\lambda}^k$ given by the kernel 
       \[
       \mathscr{A}(x,y,\lambda) = \frac{1}{(2\pi)^n}p_{\kappa, \tau}\int_{T^*_{z_\tau}M} e^{i\varphi_\tau(x,\zeta,y)} \mathfrak{b}_k(z_\tau,\eta,\lambda)\, d\zeta. 
       \]
   \end{defn}
   \begin{rem}
       Notice that these operators are not the exact analogs of what an parameter-Bessel potential would be because in Subsection \ref{besselpotentials} we saw that the symbol of the Laplacian involves the scalar curvature. 
   \end{rem}
   \begin{lem}\label{lemmaBesselLamb}
   Let $s,l,k\in\bbR$ be such that $l\geq k$. Let $\mathfrak{B}_{\lambda}^k$ be the pseudo-differential operator given in Definition \ref{auxOperator}. Then 
   \begin{align*}
       \|\mathfrak{B}_{\lambda}^k\|_{H^{s}\to H^{s-l}} \leq C_{k,l} \, (1+|\lambda|^{1/|k|})^{k}, \text{ if } l\geq 0,\\
       \|\mathfrak{B}_{\lambda}^k\|_{H^{s}\to H^{s-l}} \leq C_{k,l} \,(1+|\lambda|^{1/|k|})^{-(l-k)},  \text{ if } l\leq 0.
   \end{align*}
   \end{lem}
   \begin{proof}
      Let $\mathcal{B}^s$ be the Bessel potential defined in Subsection \ref{besselpotentials}. Using the Bessel potential we can compute the norm of $\mathfrak{B}_{\lambda}^k$ as follows
\[
\|\mathfrak{B}_{\lambda}^k\|_{H^{s}\to H^{s-l}} = \| \mathcal{B}^{s-l}\mathfrak{B}_{\lambda}^k \mathcal{B}^{-s}\|_{L^2\to L^2}= \|\mathfrak{B}_{\lambda}^k\mathcal{B}^{-l}\|_{L^2\to L^2}. 
\]
The latter norm can be computed directly using local Fourier transform, yielding 
\[
\|\mathfrak{B}_{\lambda}^k\mathcal{B}^{-l}\|_{L^2\to L^2}\leq \sup_{\eta\in\bbR^n} (|\lambda|^{1/|k|} + \langle\eta\rangle_{y})^{k} \langle\eta\rangle_{y}^{-l}\leq \begin{cases}
    C_{k,l} (1+|\lambda|^{1/|k|})^{k}, &\text{ if } l\geq 0\\
    C_{k,l} \,(1+|\lambda|^{1/|k|})^{-(l-k)},  &\text{ if } l\leq 0,
\end{cases}
\]
completing the proof. 
\end{proof}
We use this lemma to first estimate the norm of the parametrix $A_\lambda^\sharp$. 
\begin{thm}\label{normParametrix}
    Let $0\leq \delta<\rho \leq 1$, $m\geq 0$, $l\geq -m$,
    $s\in\bbR$, and let $\Lambda\subset \bbC$ be a sector. Let at least one of the conditions (1)-(3) of Theorem \ref{realEstiParameSymbol} be fulfilled. Let one of the numbers $s,s-l$ be equal to zero. If $a\in S_{\rho, \delta}^{m}(\nabla)$ is parameter-elliptic with respect to $\Lambda$, then the parametrix $A_\lambda^\sharp$ (from Theorem \ref{realEstiParameSymbol}) satisfies
    \begin{align*}
       \|A_\lambda^\sharp\|_{H^{s}\to H^{s-l}} \leq C_{k,l} \, (1+|\lambda|^{1/m})^{-m}, \text{ if } l\geq 0,\\
       \|A_\lambda^\sharp\|_{H^{s}\to H^{s-l}} \leq C_{k,l} \,(1+|\lambda|^{1/m})^{-(l+m)},  \text{ if } l\leq 0.
   \end{align*}
When $m=0$, we will have that $l\geq 0$ and the inequality is understood as  
\[
\|A_\lambda^\sharp\|_{H^s\to H^{s-l}}\leq C_l (1 + |\lambda|)^{-l}, 
\]
in particular, when $l=0$ the right hand side is only a constant. 
\end{thm}
\begin{proof}
First, let us suppose $s=l=m=0$. In this case we need to prove that $A_\lambda^\sharp$ is bounded from $L^2$ to $L^2$. This follows by mimicking the Safarov's proof of Theorem \ref{L2bound}, that is, following the classical argument by H\"ormander. The argument relies on the estimates from Theorem \ref{realEstiParameSymbol}, which guarantee uniform bounds on $\lambda$, ensuring that the final constant $C$ is independent on $\lambda$. Now for the general case we take advantage of the operators $\mathfrak{B}_{\lambda}^{m}$ defined previously, so we have 
\[
\|A_\lambda^\sharp\|_{H^{s}\to H^{s-l}} =\begin{cases}
    \|\mathfrak{B}_{\lambda}^{-m}(\mathfrak{B}_{\lambda}^{m}A_\lambda^\sharp)\|_{H^{s}\to H^{s-l}}\leq \|\mathfrak{B}_{\lambda}^{-m}\|_{H^{s}\to H^{s-l}}\|\mathfrak{B}_{\lambda}^{m}A_\lambda^\sharp\|_{H^{s}\to H^{s}}\\
    \|(A_\lambda^\sharp\mathfrak{B}_{\lambda}^{m})\mathfrak{B}_{\lambda}^{-m}\|_{H^{s}\to H^{s-l}} \leq \|\mathfrak{B}_{\lambda}^{m}A_\lambda^\sharp\|_{H^{s-l}\to H^{s-l}} \|\mathfrak{B}_{\lambda}^{-m}\|_{H^{s}\to H^{s-l}}.
\end{cases} 
\]
Because of Lemma \ref{lemmaBesselLamb} and the hypothesis that $s$ or $s-l$ is zero, the only remaining inequality to prove is 
\[
\|\mathfrak{B}_{\lambda}^{m}A_\lambda^\sharp\|_{L^2\to L^2}\leq C. 
\]
Notice that because of Theorem \ref{realEstiParameSymbol}, $A_\lambda^\sharp$ is a pseudo-differential operator of order $-m$, so that the operator $\mathfrak{B}_{\lambda}^{m}A_\lambda^\sharp$ would be some zero order operator (for a fixed $\lambda$) and thus this again follows by Safarov's argument.     
\end{proof}
\begin{rem}Notice that if $\delta=0$, one can remove of the conditions $s$ or $s-l$ equal to zero. In fact, in this case the constants 
\[
C_{\alpha,q}(\lambda) = \sup_{(x,\eta)}\, (|\lambda|^{1/m} + \langle\eta\rangle_{y})^{-m(k+1)}\langle\eta\rangle_{y}^{\delta q-\rho|\alpha|}
\]
are bounded as $|\lambda|$ goes to infinity. Consequently, the Sobolev norms can be estimated, for instance, by means of Schur’s lemma.
\end{rem}
So far we have obtained that for a parameter-elliptic symbol $a\in S_{\rho, \delta}^{m}(\nabla)$, there exits a parametrix $A_\lambda^\sharp$ for the operator $a(X,D)-\lambda I$ such that the operator
\begin{equation}\label{opError}
    R_\lambda:= A_\lambda^\sharp(a(X,D)-\lambda I)-I
\end{equation} is smoothing. We will also need  estimates of Sobolev norms of this operator to finish our approximation of the resolvent. 
\begin{thm}\label{EstSoboErro}
     Let $0\leq \delta<\rho \leq 1$, $m\geq 0$, and let $\Lambda\subset \bbC$ be a sector. Let at least one of the conditions (1)-(3) of Theorem \ref{realEstiParameSymbol} be fulfilled. Let $s,t\in\bbR$ such that $t\leq s$. If $a\in S_{\rho, \delta}^{m}(\nabla)$ is parameter-elliptic with respect to $\Lambda$, then the operator $R_\lambda$ given by \eqref{opError} satisfies for any $N>0$
    \[
    \|R_\lambda\|_{H^{s}\to H^{t}} \leq C_{m,s,t} \,\frac{1}{(1+|\lambda|)^{N}}, 
    \]
    i.e. the operator $R_\lambda$ is smoothing in a parameter-sense. 
\end{thm}
\begin{proof}
    Once again we use the Bessel operators $\mathcal{B}^k$ to calculate these norms, hence 
    \[
     \|R_\lambda\|_{H^{s}\to H^{t}} =  \|\mathcal{B}^t R_\lambda \mathcal{B}^{-s}\|_{L^2\to L^2} = \|\mathcal{B}^{t-s}R_\lambda \|_{L^2\to L^2}.
    \]
    From the proof of Theorem \ref{realEstiParameSymbol} and following a similar procedure we did for the estimate \eqref{estimadoError} we obtain that that the symbol of $R_\lambda$ satisfies the following inequalities for any $N>0$
    \begin{equation}\label{estResolIntermedioPrueba}
       |\sigma_{R_\lambda}(y,\eta,\lambda)| \leq C_{K} (|\lambda|^{1/m} + \langle\eta\rangle_{y})^{-m}\langle\eta\rangle_{y}^{m-N}.
    \end{equation}
    Thus the result follows from a local calculation using Fourier transform as in the case of Lemma \ref{lemmaBesselLamb}. Indeed, we have that 
\begin{align*}
    \|\mathcal{B}^{t-s}R_\lambda \|_{L^2\to L^2}\leq \sup_{\eta\in\bbR^n}  C_{K} (|\lambda|^{1/m} + \langle\eta\rangle_{y})^{-m}\langle\eta\rangle_{y}^{m-N} \langle\eta\rangle_{y}^{t-s}&\leq C_{s,t} \,(1+|\lambda|^{1/m})^{-N+t-s}\\
    &\leq  C_{m,s,t} \,(1+|\lambda|)^{-N},
\end{align*}
where in the last inequality we used that $t\leq s$. When $m=0$, instead of \eqref{estResolIntermedioPrueba} we will have the following estimate for any $N>0$
\[
|\sigma_{R_\lambda}(y,\eta,\lambda)| \leq C_{K} (|\lambda| + \langle\eta\rangle_{y})^{-1}\langle\eta\rangle_{y}^{-N},
\]
and the result follows analogously. 
\end{proof}
At this point we have collected all the tools we needed to estimate the norm of the resolvent operator $(a(X,D)-\lambda I)^{-1}$. 
\begin{thm}\label{EstimadoFinalResolvente}
   Let $0\leq \delta<\rho \leq 1$, $m\geq 0$, $l\geq -m$,
    $s\in\bbR$, and let $\Lambda\subset \bbC$ be a sector. Let at least one of the conditions (1)-(3) of Theorem \ref{realEstiParameSymbol} be fulfilled. Let one of the numbers $s,s-l$ be equal to zero. If $a\in S_{\rho, \delta}^{m}(\nabla)$ is parameter-elliptic with respect to $\Lambda$, then the operator $a(X,D)-\lambda I$ is invertible in $H^s$ for $|\lambda|$ sufficiently large and $\lambda$ in the sector $\Lambda$. Moreover, the following estimates hold 
    \begin{align*}
       \|(a(X,D)-\lambda I)^{-1}\|_{H^{s}\to H^{s-l}} \leq C_{m,l} \, (1+|\lambda|^{1/m})^{-m}, \text{ if } l\geq 0,\\
       \|(a(X,D)-\lambda I)^{-1}\|_{H^{s}\to H^{s-l}} \leq C_{m,l} \,(1+|\lambda|^{1/m})^{-(l+m)},  \text{ if } l\leq 0.
   \end{align*}
\end{thm}
\begin{proof}
    Let us assume $s=0$, the case $s+l=0$ is analogous. Let $|\lambda|$ be sufficiently large so that by Theorem \ref{EstSoboErro} we have that 
    \begin{equation}\label{EstGeometricaInversa}
        \|R_\lambda\|_{L^2\to L^2}<\frac{1}{2} \text{ and } \|R_\lambda\|_{H^l\to H^l}<\frac{1}{2}.
    \end{equation}
    Thus the operator 
    \[
    \sum_{j=0}^\infty (-R_\lambda)^j A_\lambda^\sharp 
    \]
    defines a left inverse of $a(X,D)-\lambda I$ on $L^2$ since 
    \[
    (I- (-R_\lambda))^{-1} = \sum_{j=0}^\infty (-R_\lambda)^j \text{ and } A_\lambda^\sharp(a(X,D)-\lambda I)= I+ R_\lambda.
    \]
    Furthermore, the estimates for the Sobolev norms follow from \eqref{EstGeometricaInversa} and Theorem \ref{normParametrix} since 
    \[
    \sum_{j=0}^\infty \|(-R_\lambda)^j\|_{H^l\to H^l}<1 \text{ and } \|A_\lambda^\sharp\|_{L^2\to H^{l}} \leq \begin{cases}
    C_{m,l} (1+|\lambda|^{1/m})^{-m}, &\text{ if } l\geq 0\\
    C_{m,l} \,(1+|\lambda|^{1/m})^{-(l+m)},  &\text{ if } l\leq 0.\end{cases}
    \]
    Similarly we get that such operator is also a right inverse, completing the proof. 
\end{proof}
\begin{rem}
    Notice that the last theorem implies, in particular, that for $|\lambda|$ sufficiently large, we have 
    \[
    \|(a(X,D)-\lambda I)^{-1}\|_{L^{2}\to L^{2}} \leq C \,\frac{1}{|\lambda|},
    \] 
    as we promised at the beginning of Subsection \ref{parameter-ellipticity}. 
\end{rem}
    
Finally, we conclude this subsection by proving that indeed the parametrix $A_\lambda^\sharp$ approximates the resolvent operator.  
\begin{cor}\label{normResolvent}
     Let $0\leq \delta<\rho \leq 1$, $m\geq 0$, $l\geq -m$, $s\in\bbR$, and let $\Lambda\subset \bbC$ be a sector. Let at least one of the conditions (1)-(3) of Theorem \ref{realEstiParameSymbol} be fulfilled. Let one of the numbers $s,s-l$ be equal to zero. If $a\in S_{\rho, \delta}^{m}(\nabla)$ is parameter-elliptic with respect to $\Lambda$, then for any $N>0$, we have 
    \[
    \|A_\lambda^\sharp - (a(X,D)-\lambda I)^{-1}\|_{H^{s}\to H^{s-l}} \leq C_{m,l} \,\frac{1}{(1+|\lambda|)^{N}}, 
    \] 
    i.e. the operator $A_\lambda^\sharp - (a(X,D)-\lambda I)^{-1}$ is smoothing in a parameter-sense. 
\end{cor}
\begin{proof}
    Notice that we can factor out the desired difference of operators in the following way:
    \[
    A_\lambda^\sharp - (a(X,D)-\lambda I)^{-1} = (A_\lambda^\sharp(a(X,D)-\lambda I) - I)(a(X,D)-\lambda I)^{-1}.
    \]
    Hence by Theorem \ref{EstSoboErro} and Theorem \ref{EstimadoFinalResolvente} we have that 
    \begin{align*}
        \|A_\lambda^\sharp - (a(X,D)-&\lambda I)^{-1}\|_{H^{s}\to H^{s-l}} \\
        &\leq \|A_\lambda^\sharp(a(X,D)-\lambda I) - I\|_{H^{s-l}\to H^{s-l}}\|(a(X,D)-\lambda I)^{-1}\|_{H^{s}\to H^{s-l}}\\
        &= \|R_\lambda\|_{H^{s-l}\to H^{s-l}}\|(a(X,D)-\lambda I)^{-1}\|_{H^{s}\to H^{s-l}}\\
        &\leq  C_{m,l} \,\frac{1}{(1+|\lambda|)^{N}}\cdot \begin{cases}
             (1+|\lambda|^{1/m})^{-m}, &\text{ if } l\geq 0\\
    (1+|\lambda|^{1/m})^{-(l+m)},  &\text{ if } l\leq 0
        \end{cases} \\
        & \leq  C_{m,l} \,\frac{1}{(1+|\lambda|)^{N}}
    \end{align*}
    for any $N>0$, which  finishes the proof. 
\end{proof}

\section{Holomorphic functional calculus}\label{Section 4}
In this section, we develop the holomorphic functional calculus for operators in the classes $\Psi_{\rho, \delta}^m\left(\Omega^\kappa, \nabla\right)$ using the results from the previous section. Once this framework is established, in Section \ref{Section 5}, we turn our attention to key applications of our main result. 

Let $A$ be a parameter-elliptic pseudo-differential operator in $\Psi_{\rho, \delta}^{m}\left(\Omega^{\kappa}, \nabla\right)$ with respect to a sector $\Lambda\subset\bbC$ on a closed Riemannian manifold $(M,g)$. Then by Theorem \ref{EstimadoFinalResolvente} the resolvent $(A-\lambda I)^{-1}$ is well defined for $\Lambda$. Moreover, Theorem \ref{EstimadoFinalResolvente} together with a standard resolvent formula and the theory of compact operators imply that the spectrum of $A$ is discrete, so that we can find a ray of minimal growth $R_\theta$ (see Definition \ref{agmonAngle}) in the sector $\Lambda$ such that $\sigma(A)\cap R_\theta = \emptyset$ or $\sigma(A)\cap R_\theta = \{0\}$. Henceforth, for simplification, we assume the following: 
\begin{enumerate}
    \item zero is not in the spectrum of $A$, so that actually there is not a point of the spectrum inside $|\lambda|<\epsilon$;
    \item the ray of minimal growth is $R_\theta=(-\infty,0]$.
\end{enumerate}
\begin{rem}
  We emphasise that we are not losing generality under these assumptions because always we can consider the operators $A+\varepsilon$ or $e^{i\theta}A$ instead of $A$.  
\end{rem}
Therefore by Remark \ref{extensionToSmallBall} we have that $(A-\lambda I)^{-1}$ is well-defined for $\Lambda_\epsilon$. Thus, for a good enough function $f$ holomorphic on $\bbC\setminus \Gamma$, where $\Gamma$ is de boundary of $\Lambda_\epsilon$, we can define the operator $f(A)$ using a Dunford-Riesz integral along the ray of minimal growth:
\begin{equation}\label{functionalCalculus}
    f(A) = \frac{1}{2\pi i}\int_{\Gamma} f(\lambda) (A-\lambda I)^{-1}\, d\lambda. 
\end{equation} 
Notice that, by Cauchy's integral formula, this definition is invariant under homotopic deformations of the contour. Hence, we can deform $\Gamma$ to a curve that simplifies the computations. Indeed, let $\Gamma = \Gamma_1\cup\Gamma_2\cup\Gamma_3$ be the following contour
\begin{equation}\label{contorno}
    \Gamma = \begin{cases}
    re^{i\pi} \,(+\infty>r>\epsilon)\text{ on } \Gamma_1,\\
    \epsilon e^{i\theta}\, (\pi>\theta>-\pi) \text{ on } \Gamma_2,\\
    re^{-i \pi}\, (\epsilon<r<+\infty)\text{ on } \Gamma_3.
\end{cases}
\end{equation}

\begin{center}
    \begin{tikzpicture}
    \draw[help lines, color=gray!30, dashed] (-3.9,-1.9) grid (1.9,1.9);
\draw [->] (-4,0)--(2,0) node[right]{$\text{Re}$};
\draw [->] (0,-2)--(0,2) node[above]{$\text{Im}$};
  \draw[-,blue, thick] (-4,0)--(0,0) node[xshift=-1cm, below]{$\resizebox{.02\hsize}{!}{$R_\theta$}$};

    % draw the two lines
    \begin{scope}[even odd rule, decoration={
    markings,
    mark=at position 0.5 with {\arrow{>}}}]
        \clip (2,-2) rectangle (-4,2) (0,0) circle (1);
        \draw [postaction={decorate},thick] (-4,0.5) -- (0,0.5) node[xshift=-3cm, above]{$ \resizebox{.02\hsize}{!}{$\Gamma_1$} $};
        \draw [postaction={decorate},thick] (0,-0.5) -- (-4,-0.5) node[xshift=1cm, below]{$\resizebox{.02\hsize}{!}{$\Gamma_3$}$};
        
    \end{scope}

    % draw the circle
    \begin{scope}[even odd rule, decoration={
    markings,
    mark=at position 0.15 with {\arrow{<}},
    mark=at position 0.85 with {\arrow{<}}}]
        \clip (-2,-2) rectangle (3,2) (0,0.5) rectangle (-3,-0.5);
        \draw [postaction={decorate}, thick] (0,0) circle (1) node[xshift=1.2cm, above]{$\resizebox{.02\hsize}{!}{$\Gamma_2$}$};
    \end{scope}
\end{tikzpicture}
\end{center}

Under our assumptions, this contour is precisely the boundary of our sector $\Lambda_\epsilon$, so from now on it is fixed. 

We note that the operator \eqref{functionalCalculus} is well-defined on $L^2$ due to the estimates established in Theorem \ref{EstimadoFinalResolvente}. Our primary interest is in determining when $f(A)$ remains a pseudo-differential operator. To this end, it is essential to analyze whether it admits a well-defined symbol belonging to the global symbol classes $S_{\rho, \delta}^{m}(\nabla)$. With this goal in mind, in the previous section, we constructed a parametrix $A_\lambda^\sharp$ that approximates the resolvent of $A$. The explicit formulation of this parametrix, along with symbolic estimates, will play a crucial role in investigating the structure of $f(A)$. In particular, we can now rewrite \eqref{functionalCalculus} as follows:
\begin{align}\label{approxResolventOperator}
\begin{split}
       f(A) &= \frac{1}{2\pi i}\int_{\Gamma} f(\lambda) (A-\lambda I)^{-1}\, d\lambda\\
       &= \frac{1}{2\pi i}\int_{\Gamma} f(\lambda)A_\lambda^\sharp \, d\lambda + \frac{1}{2\pi i}\int_{\Gamma} f(\lambda) \left[(A-\lambda I)^{-1}- A_\lambda^\sharp\right]\, d\lambda,    
\end{split}
   \end{align}
   where the last second term is a smoothing operator due to Corollary \ref{normResolvent}, thus we can just focus on studying the first term. Indeed, let us recall from previous section that the parametrix $A_\lambda^\sharp$ has a symbol $a^\sharp$ such that 
\[
   a^\sharp(y,\eta,\lambda):=a_{A^\sharp}(y,\eta,\lambda) \sim (a(y,\eta)-\lambda)^{-1} \sum_{k=1}^\infty b_{k}(y,\eta, \lambda) \, \text{ as } \langle\eta\rangle_{y} \to\infty,
   \]
   where for all $\lambda$ each $b_k\in S_{\rho,\delta}^{-kr}(\nabla)$ for some positive $r$, therefore we quickly identify that the symbol of the operator \eqref{functionalCalculus} would be given by the formula
   \begin{equation}\label{functCalcSymbol}
        a_{f(a)}(y,\eta) = \frac{1}{2\pi i}\int_{\Gamma} f(\lambda) a^\sharp(y,\eta,\lambda)\, d\lambda. 
   \end{equation}
   So, as in the previous section, it will be useful to study first the expression related to $(a(y,\eta)-\lambda)^{-1}$ in order to obtain information about \eqref{functCalcSymbol}. What we mean is that it will be convenient to start with the function 
   \begin{equation}\label{simboloAuxiliar}
       \tilde{a}_{f(a)}(y,\eta) = \frac{1}{2\pi i}\int_{\Gamma} f(\lambda) (a(y,\eta)-\lambda)^{-1}\, d\lambda
   \end{equation}
   instead of directly with \eqref{functCalcSymbol}. 
\begin{thm}\label{thmFunctionalCal}
 Let $0\leq \delta<\rho \leq 1$, $m> 0$. Suppose that at least one of the following conditions is fulfilled:
\begin{enumerate}
    \item $\rho>\frac{1}{2}$;
    \item the connection $\nabla$ is symmetric and $\rho>\frac{1}{3}$;
    \item the connection $\nabla$ is flat.
\end{enumerate}
Let $a\in S_{\rho, \delta}^{m}(\nabla)$ be parameter-elliptic with respect to $\Lambda$. Suppose $f$ is holomorphic on $\bbC\setminus\Lambda$ and satisfies 
 \begin{equation}\label{decaiminetoFuncion}
      |f(\lambda)|\leq C |\lambda|^s
 \end{equation}
 uniformly for some $s<0$. Then formula \eqref{simboloAuxiliar} defines a symbol $\tilde{a}_{f(A)}\in S_{\rho, \delta}^{ms}(\nabla)$ for large $\eta$. Moreover, the operator $f(a(X,D))$ given by formula \eqref{functionalCalculus} is well-defined, its symbol $a_{f(a)}$ belongs to $S_{\rho, \delta}^{ms}(\nabla)$ and satisfies 
   \begin{align}\label{asympSymbolFuncOpe}
       \begin{split}
           \sigma_{f(a)}(y,\eta) &\sim \frac{1}{2\pi i}\int_{\Gamma} f(\lambda) a^\sharp(y,\eta,\lambda)\, d\lambda\\
           &\sim f(a(y,\eta)) + \sum_{k=1}^\infty \sum_{\mathfrak{I}_k}(-1)^{1+|\mathfrak{I}_{k,L}|}  \frac{f^{(k+|\mathfrak{I}_{k,L}|)}(a(y,\eta))}{(k+|\mathfrak{I}_{k,L}|)!}\mathfrak{r}_{\mathfrak{I}_k}(y,\eta)
       \end{split}
   \end{align}
   as $\langle\eta\rangle_{y}\to\infty$, where the notation $\mathfrak{I}_k$ is as in Remark \ref{remExpAsyPara}. 
\end{thm}
\begin{proof}
First, let us assume that $-1<s<0$. We check that \eqref{simboloAuxiliar} defines a symbol for large $\eta$. Notice that 
\begin{align*}
    | \tilde{a}_{f(a)}(y,\eta)| &\leq \frac{1}{2\pi}\int_{\Gamma} |f(\lambda) (a(y,\eta)-\lambda)^{-1}|\, |d\lambda|\\
    &\leq C_K \int_{\Gamma} |\lambda|^s (|\lambda|^{1/m} + \langle\eta\rangle_{y})^{-m} \, |d\lambda|. 
\end{align*}
The latter integral splits on the 3 different curves: 
\begin{align*}
    \int_{\Gamma} |\lambda|^s (|\lambda|^{1/m} + \langle\eta\rangle_{y})^{-m} \, |d\lambda| &= 2\int_{\epsilon}^\infty r^s(r^{1/m}+\langle\eta\rangle_{y})^{-m}\, dr + \int_{-\theta_0}^{\theta_0} \epsilon^s(\epsilon^{1/m}+\langle\eta\rangle_{y})^{-m}\, d\theta\\
    &\leq C \langle\eta\rangle_{y}^{ms} \int_{\epsilon}^\infty r^{s-1}\, dr + \epsilon^s \langle\eta\rangle_{y}^{-m}\\
    &\leq C \langle\eta\rangle_{y}^{ms},
\end{align*}
where in the last step we used that $-1<s<0$. So
\[
| \tilde{a}_{f(a)}(y,\eta)| \leq C_K \langle\eta\rangle_{y}^{ms}. 
\]
By induction, following a similar calculation as in Section \ref{Section 3} and using Theorem \ref{ParaEllEst} we obtain for all $\alpha, i_1, \ldots, i_q$ 
    \[
        |\partial_{\eta}^{\alpha} \nabla_{i_{1}} \ldots \nabla_{i_{q}}  \tilde{a}_{f(a)}(y,\eta)| \leq C_{K, k, \alpha, i_1, \ldots, i_q} \langle\eta\rangle_{y}^{ms+\delta q -\rho |\alpha|}, 
    \]
    implying that $\tilde{a}_{f(a)} \in S_{\rho, \delta}^{ms}(\nabla)$ for large $\eta$. Therefore, an analogous argument but using Theorem \ref{realEstiParameSymbol} gives us that for $-1<s<0$ and for all $\alpha, i_1, \ldots, i_q$ 
    \[
        |\partial_{\eta}^{\alpha} \nabla_{i_{1}} \ldots \nabla_{i_{q}}  a_{f(a)}(y,\eta)| \leq C_{K, k, \alpha, i_1, \ldots, i_q} \langle\eta\rangle_{y}^{ms+\delta q -\rho |\alpha|}, 
    \]
    i.e. $a_{f(a)}$ defines a symbol in $S_{\rho, \delta}^{ms}(\nabla)$, where $a_{f(a)}$ is given by \eqref{functCalcSymbol}. 
    
    Now let us deal with the general case $s<0$. The idea is to come back to the situation $-1<s<0$. Let $k$ be an integer such that $k< s<0$, thus there exists a holomorphic function $g$ on $\bbC\setminus\Gamma$ such that $f(\lambda)=g(\lambda)^{-k}$. Therefore this new function satisfies $|g(\lambda)|\leq |\lambda|^{s/(-k)}$ with $-1<\frac{s}{-k}<0$, so by the previous procedure the formula
    \[
    a_{g(a)}(y,\eta) = \frac{1}{2\pi i}\int_{\Gamma} g(\lambda) a^\sharp(y,\eta,\lambda)\, d\lambda
   \]
   defines a symbol in $S_{\rho, \delta}^{ms/(-k)}(\nabla)$. By construction we have that 
   \[
   a_{f(a)}(y,\eta) = a_{g(a)^{-k}}(y,\eta),
   \]
   and the symbol $a_{g(a)^{-k}}$ corresponds to the operator $g(a(X,D))^{-k}$, i.e. the pseudo-differential operator $g(a(X,D))$ composed with itself $-k$ times. Hence applying Theorem \ref{producto} $-k$ times, which is possible because we have just proven that $a_{g(a)}(y,\eta)\in S_{\rho, \delta}^{ms/(-k)}(\nabla)$, we get that $a_{f(a)} = a_{g(a)^{-k}}\in S_{\rho, \delta}^{ms}(\nabla)$. 
   
   The validity of the asymptotic expansion
   \[
   \sigma_{f(a)}(y,\eta) \sim \frac{1}{2\pi i}\int_{\Gamma} f(\lambda) a^\sharp(y,\eta,\lambda)\, d\lambda  \, \text{ as } \langle\eta\rangle_{y}\to\infty,
   \]
    follows immediately from \eqref{approxResolventOperator} because the operator 
   \[
   \frac{1}{2\pi i}\int_{\Gamma} f(\lambda) \left[(A-\lambda I)^{-1}- A_\lambda^\sharp\right]\, d\lambda
   \]
   is smoothing. Moreover, putting this together with the expansion \eqref{asymPara2} and the Cauchy integral theorem we obtain
   \[
   \sigma_{f(a)}(y,\eta) \sim f(a(y,\eta)) + \sum_{k=1}^\infty \sum_{\mathfrak{I}_k}(-1)^{1+|\mathfrak{I}_{k,L}|}  \frac{f^{(k+|\mathfrak{I}_{k,L}|)}(a(y,\eta))}{(k+|\mathfrak{I}_{k,L}|)!}\mathfrak{r}_{\mathfrak{I}_k}(y,\eta)
   \]
   as $\langle\eta\rangle_{y}\to\infty$, concluding the proof. 
\end{proof}
\begin{rem}
    Notice that in the previous proof we chose an arbitrary integer $k$ to rewrite the function $f$ as the $k$-th power of another function, which indeed the result is independent of this selection.  
\end{rem}
\begin{cor}\label{corFunctionalPositive}
    The results from Theorem \ref{thmFunctionalCal} remain true if we allow the condition \eqref{decaiminetoFuncion} to hold for $s\geq 0$. 
\end{cor}
\begin{proof}
    In fact, let $k$ be an integer such that $0\leq s< k$, then we rewrite $f$ as $f(\lambda)=g(\lambda)\lambda^k$ where $g(\lambda):= f(\lambda)\lambda^{-k}$ and $|g(\lambda)|\leq C|\lambda|^{s-k}$. We get the result from Theorems \ref{thmFunctionalCal} and \ref{producto}. This procedure is independent of $k$ due to Cauchy integral theorem.
\end{proof}
Just to summarize, we have obtain that $f(A)$ is again a pseudo-differential operator whenever $f$ is of polynomial decay or polynomial growth. 

We finish this subsection by highlighting that for zero order pseudo-differential operators one can do better than Theorem \ref{thmFunctionalCal}, which means that we can remove the decay or growth condition on $f$. This is possible because due to Theorem \ref{producto} the zero order operators are bounded from $L^2$ to $L^2$, so that their spectrum is not only discrete but bounded. Thus in the integral \eqref{functionalCalculus} an open curve is not needed but a finite curve surrounding the spectrum of the operator $A$. Hence by repeating previous arguments, but with the definition of parameter-ellipticity for zero order pseudo-differential operators, the corresponding results in Subsection \ref{parameter-ellipticity} and the Cauchy's integral theorem we obtain the following result: 
\begin{thm}\label{thmFunctionalCalZero}
 Let $0\leq \delta<\rho \leq 1$. Let at least one of the conditions (1)-(3) of Theorem \ref{thmFunctionalCal} be fulfilled. Let $a\in S_{\rho, \delta}^{0}(\nabla)$, let $\Gamma$ be a finite closed curve enclosing the spectrum of $a(X,D)$, and let $\Lambda$ denote the complementary region, i.e., the region lying outside $\Gamma$. Then $a$ is parameter-elliptic with respect to $\Lambda$ and for any holomorphic function $f$ on the spectrum of $a(X,D)$ we have that 
 \[
 f(A) = \frac{1}{2\pi i}\int_{\Gamma} f(\lambda) (A-\lambda I)^{-1}\, d\lambda
 \]
 defines a pseudo-differential operator in the class $\Psi_{\rho, \delta}^{0}\left(\Omega^{\kappa}, \nabla\right)$. Furthermore, the following asymptotic expansion holds
 \begin{equation}\label{asymptForFunctionOfSymbol}
     \sigma_{f(a)}(y,\eta) \sim f(a(y,\eta)) + \sum_{k=1}^\infty \sum_{\mathfrak{I}_k}(-1)^{1+|\mathfrak{I}_{k,L}|}  \frac{f^{(k+|\mathfrak{I}_{k,L}|)}(a(y,\eta))}{(k+|\mathfrak{I}_{k,L}|)!}\mathfrak{r}_{\mathfrak{I}_k}(y,\eta)
   \end{equation}
   as $\langle\eta\rangle_{y}\to\infty$, where the notation $\mathfrak{I}_k$ is as in Remark \ref{remExpAsyPara}. 
\end{thm}

\section{Examples} \label{examplesSubsect}
We now proceed to present concrete and illustrative examples of operators of the form $f(A)$. Before doing so, however, we introduce a well-behaved family of parameter-elliptic pseudo-differential operators. This will provide a rich collection of candidates to which we can apply the functional calculus developed above.
\begin{defn}
    Let $X$ be a Banach space. We say that an operator $T: D(T)\subset X\to X$ is \textit{positive real} if $\sigma(T)\subseteq \{z\in\bbC: \operatorname{Re}z\geq 0\}$. In particular, if $X$ is a Hilbert space this condition is equivalent to 
    \[
    \operatorname{Re}\,\langle Tx,x\rangle \geq 0,
    \]
    for all $x\in X$. 
\end{defn}
We point out that, of course, this class of operators includes the self-adjoint and positive operators. 
\begin{prop}
    Let $0\leq \delta<\rho \leq 1$ and $m\geq 0$. Let $A$ be an elliptic (i.e. $A\in HS_{\rho, \delta}^{m,m}(\nabla)$) positive real pseudo-differential operator, then there exist a sector $\Lambda\subset\bbC$ such that $A$ is parameter-elliptic with respect to $\Lambda$. 
\end{prop}

\begin{proof}
In fact, we will prove that there is a family of sectors doing the trick. Let $3\pi/4<\theta_0<\pi$, and let $\Lambda_{\theta_0} = \{z\in \bbC: |\operatorname{arg} z|>\theta_0\}$. We will prove that $A$ is parameter-elliptic with respect to $\Lambda_{\theta_0}$, i.e., that there exist a constant $c_K>0$ such that 
\[
         (|\lambda|^{1/m} + \langle\eta\rangle_{y})^{m} \leq \mathrm{const}_{K} \left|a(y,\eta)-\lambda \right|,
        \]
for all $\lambda\in\Lambda_{\theta_0}$ and $(y,\eta)\in T^*M \text{ with } y\in K, |\lambda|+\langle\eta\rangle_{y}\geq c'_K$. 

Since $A$ is elliptic, there is a constant $c'_K>0$ such that 
        \[
         \langle\eta\rangle_{y}^{m} \leq \mathrm{const'}_{K} \left|a(y,\eta)\right|,
        \]
for all $(y,\eta)\in T^*M \text{ with } y\in K, \langle\eta\rangle_{y}\geq c_K$. We fix $c_K:= c'_K$, let us proceed with the estimation process by considering two different regions:
\begin{itemize}
    \item Suppose $|\lambda|\leq \frac{1}{2 \mathrm{const'}_{K}} \langle\eta\rangle_{y}^{m}$. Then in this region ellipticity allows us to conclude that $|a(y,\eta)|-|\lambda|\geq 0$. By the triangle inequality, ellipticity and the latter observation we get 
    \begin{align*}
        |a(y,\eta)-\lambda|\geq ||a(y,\eta)|-|\lambda|| = |a(y,\eta)|-|\lambda|&\geq \frac{1}{ \mathrm{const'}_{K}} \langle\eta\rangle_{y}^{m}- \frac{1}{2 \mathrm{const'}_{K}} \langle\eta\rangle_{y}^{m}\\
        &= \frac{1}{2 \mathrm{const'}_{K}} \langle\eta\rangle_{y}^{m}. 
    \end{align*}
    Now, because $|\lambda|+\langle\eta\rangle_{y}\geq c_K$ we further obtain that 
    \begin{align*}
        \frac{1}{2 \mathrm{const'}_{K}} \langle\eta\rangle_{y}^{m} = \frac{1}{4 \mathrm{const'}_{K}} \langle\eta\rangle_{y}^{m}+\frac{1}{4 \mathrm{const'}_{K}} \langle\eta\rangle_{y}^{m}&\geq C(|\lambda|+\langle\eta\rangle_{y}^{m})\\
        &\approx C_1 (|\lambda|^{1/m} + \langle\eta\rangle_{y})^{m}. 
    \end{align*}
    Therefore we have obtained that  
    \[
    (|\lambda|^{1/m} + \langle\eta\rangle_{y})^{m} \leq \frac{1}{C_1} \left|a(y,\eta)-\lambda \right|. 
    \]
    \item Suppose $|\lambda|\geq \frac{1}{2 \mathrm{const'}_{K}} \langle\eta\rangle_{y}^{m}$. Notice that for any $\lambda\in\Lambda_{\theta_0}$ it holds that $\operatorname{Re}\lambda<0$, and because $A$ is positive real we have that $\operatorname{Re}a(y,\eta)\geq 0$. Consequently, we have the following inequalities 
    \begin{align*}
        |a(y,\eta)-\lambda| = \sqrt{(\operatorname{Re}a(y,\eta)-\operatorname{Re}\lambda)^2+(\operatorname{Im}a(y,\eta)-\operatorname{Im}\lambda)^2}\geq \sqrt{(-\operatorname{Re}\lambda)^2}\\
        = |\operatorname{Re}\lambda|>|\lambda||\cos\theta_0|.
    \end{align*}
    Again, since $|\lambda|+\langle\eta\rangle_{y}\geq c_K$ we further obtain that
    \begin{align*}
        |\lambda||\cos\theta_0|=|\cos\theta_0|\left(\frac{|\lambda|}{2}+\frac{|\lambda|}{2} \right)&\geq C |\cos\theta_0| (|\lambda|+\langle\eta\rangle_{y}^{m})\\
        & \approx C_2 |\cos\theta_0| (|\lambda|^{1/m} + \langle\eta\rangle_{y})^{m}.
    \end{align*}
    So we have proved that 
    \[
    (|\lambda|^{1/m} + \langle\eta\rangle_{y})^{m} \leq \frac{1}{C_2 |\cos\theta_0|} \left|a(y,\eta)-\lambda \right|. 
    \]
\end{itemize}
Finally, we combine both estimates by defining 
\[
\mathrm{const}_{K}:= \max \left\{\frac{1}{C_1}, \frac{1}{C_2 |\cos\theta_0|}\right\}.
\]
Hence it holds that
\[
(|\lambda|^{1/m} + \langle\eta\rangle_{y})^{m} \leq  \mathrm{const}_{K} \left|a(y,\eta)-\lambda \right|,
\]
completing the proof. 
\end{proof}
Let us recall that the sector $\Lambda$ is fixed and is the one surrounded by the contour \eqref{contorno}. Our first example are the complex powers of a pseudo-differential operator. 
\begin{cor}\label{ComplexPow}
    Let $0\leq \delta<\rho \leq 1$, $m\geq 0$. Let at least one of the conditions (1)-(3) of Theorem \ref{thmFunctionalCal} be fulfilled. Let $a\in S_{\rho, \delta}^{m}(\nabla)$ be parameter-elliptic with respect to $\Lambda$. If the symbol $a$ satisfies 
    \[
    \exp (\log( a)) \equiv a,
    \]
    then $\sigma_{a^s}(y,\eta)$ defines a symbol in $S_{\rho, \delta}^{m\operatorname{Re}(s)}(\nabla)$ for any $s\in \bbC$, and satisfies 
    \begin{align*}
        \sigma_{a^s}(y,\eta) \sim &\, a(y, \eta)^s\\
        +& \sum_{k=1}^\infty \sum_{\mathfrak{I}_k}(-1)^{1+|\mathfrak{I}_{k,L}|}s(s-1)\cdots(s-k-|\mathfrak{I}_{k,L}|+1) \frac{a(y,\eta)^{s-k-|\mathfrak{I}_{k,L}|}}{(k+|\mathfrak{I}_{k,L}|)!}\mathfrak{r}_{\mathfrak{I}_k}(y,\eta)
    \end{align*}
    as $\langle\eta\rangle_{y}\to\infty$, where 
    \[
    a(y, \eta)^s:= \exp (s\log(a(y,\eta))), \quad (y,\eta)\in T^*M. 
    \]
    Moreover, the family of operators $A^s:=a(X,D)^s$ is a group, $A^{s+t}=A^sA^t$, $A^1=A$, $A^0=I$. Furthermore, $A^s$ is an elliptic, invertible pseudo-differential operator such that for any real $k$, $A^s$ is an isomorphism from $H^k$ to $H^{k-m\operatorname{Re}(s)}$. For each real $\nu$, $A^s$ is an analytic family of operators from $H^k$ to $H^{k-\nu}$ for $\operatorname{Re}(s)<\nu$.  
\end{cor}
\begin{proof}
    Clearly the function $\lambda^s:= \exp (s\log(\lambda))$ satisfies the estimate 
    \[
    |\lambda^s|\leq |\lambda|^{\operatorname{Re}(s)},
    \]
    thus the first part follows from Theorem \ref{thmFunctionalCal} if $\operatorname{Re}(s)<0$ or from Corollary \ref{corFunctionalPositive} if $\operatorname{Re}(s)\geq0$. Now let us check the other properties by using standard arguments. Suppose $\operatorname{Re}(s),\operatorname{Re}(t) <0$, let us consider the contour $\Gamma' =\Gamma_1'\cup\Gamma_2'\cup\Gamma_3'$ given by 
\begin{equation*}\label{contorno2}
    \Gamma' = \begin{cases}
    re^{i(\pi-\varepsilon)} \,(+\infty>r>\frac{3}{2}\epsilon)\text{ on } \Gamma_1',\\
    \frac{3}{2}\epsilon e^{i\theta}\, (\pi-\varepsilon>\theta>-\pi+\varepsilon) \text{ on } \Gamma_2',\\
    re^{-i (\pi-\varepsilon)}\, (\frac{3}{2}\epsilon<r<+\infty)\text{ on } \Gamma_3'. 
\end{cases}
\end{equation*}

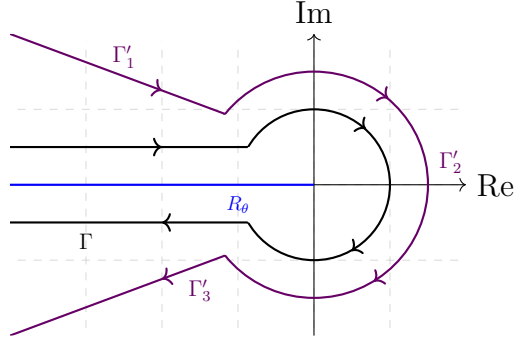
\begin{center}
    \begin{tikzpicture}
    \draw[help lines, color=gray!30, dashed] (-3.9,-1.9) grid (1.9,1.9);
\draw [->] (-4,0)--(2,0) node[right]{$\text{Re}$};
\draw [->] (0,-2)--(0,2) node[above]{$\text{Im}$};
  \draw[-,blue, thick] (-4,0)--(0,0) node[xshift=-1cm, below]{$\resizebox{.02\hsize}{!}{$R_\theta$}$};

    % draw the two lines
    \begin{scope}[even odd rule, decoration={
    markings,
    mark=at position 0.5 with {\arrow{>}}}]
        \clip (2,-2) rectangle (-4,2) (0,0) circle (1);
        \draw [postaction={decorate},thick] (-4,0.5) -- (0,0.5) node[xshift=-3cm, above]{$\resizebox{.02\hsize}{!}{ }$};
        \draw [postaction={decorate},thick] (0,-0.5) -- (-4,-0.5) node[xshift=1cm, below]{$\resizebox{.012\hsize}{!}{$\Gamma$}$};
        
    \end{scope}
    % draw the two lines 2
    \begin{scope}[even odd rule, decoration={
    markings,
    mark=at position 0.5 with {\arrow{>}}}]
        \clip (2,-2) rectangle (-4,2) (0,0) circle (1.5);
        \draw [postaction={decorate},color={rgb,255:red,102; green,0; blue,102}, thick] (-4,2) -- (0,0.5) node[xshift=-2.5cm, yshift=0.9cm, above]{$\resizebox{.02\hsize}{!}{$\Gamma_1'$} $};
        \draw [postaction={decorate},color={rgb,255:red,102; green,0; blue,102},thick] (0,-0.5) -- (-4,-2) node[xshift=2.5cm, yshift=0.9cm, below]{$\resizebox{.02\hsize}{!}{$\Gamma_3'$}$};
        
    \end{scope}

    % draw the circle
    \begin{scope}[even odd rule, decoration={
    markings,
    mark=at position 0.15 with {\arrow{<}},
    mark=at position 0.85 with {\arrow{<}}}]
        \clip (-2,-2) rectangle (3,2) (0,0.5) rectangle (-3,-0.5);
        \draw [postaction={decorate}, thick] (0,0) circle (1) node[xshift=1.2cm, above]{$\resizebox{.02\hsize}{!}{ }$};
    \end{scope}
    % draw the circle 2
    \begin{scope}[even odd rule, decoration={
    markings,
    mark=at position 0.15 with {\arrow{<}},
    mark=at position 0.85 with {\arrow{<}}}]
        \clip (-2,-2) rectangle (3,2) (0,0.93) rectangle (-3,-0.93);
        \draw [postaction={decorate},color={rgb,255:red,102; green,0; blue,102}, thick] (0,0) circle (1.5) node[xshift=1.8cm, above]{$\resizebox{.02\hsize}{!}{$\Gamma_2'$}$};
    \end{scope}
\end{tikzpicture}
\end{center}
By the Cauchy integral theorem we can change $\Gamma$ to $\Gamma'$ in \eqref{functionalCalculus} without affecting the construction. We use this fact and the identity 
\[
(A-\lambda I)^{-1} (A-\mu I)^{-1} = \frac{1}{\lambda-\mu} \left[(A-\lambda I)^{-1} - (A-\mu I)^{-1}\right]
\]
to compute the operator $A^sA^t$. In fact, we have 
\begin{align*}
    A^sA^t &= -\frac{1}{4\pi^2}\int_{\Gamma'}\int_{\Gamma}\lambda^s\mu^t (A-\lambda I)^{-1} (A-\mu I)^{-1} \, d\mu \,d\lambda\\
    &= -\frac{1}{4\pi^2}\int_{\Gamma'} \lambda^s(A-\lambda I)^{-1}\int_{\Gamma}\frac{\mu^t}{\lambda-\mu} \, d\mu \,d\lambda+\frac{1}{4\pi^2}\int_{\Gamma'}\int_{\Gamma}\frac{\lambda^s\mu^t}{\lambda-\mu}  (A-\mu I)^{-1} \, d\mu \,d\lambda\\
    &= \frac{i}{2\pi}\int_{\Gamma'}\lambda^{s+t} (A-\lambda I)^{-1} \, d\lambda + \frac{1}{4\pi^2}\int_{\Gamma}\mu^t(A-\mu I)^{-1}\underbrace{\int_{\Gamma'}\frac{\lambda^s}{\lambda-\mu} \,d\lambda}_{= 0}\, d\mu\\
    &= A^{s+t}. 
\end{align*}
Now if $\operatorname{Re}(s),\operatorname{Re}(t)\geq0$, we have that from construction $A^s= A^kA^{s-k}$ and $A^t= A^lA^{t-l}$, where $k$ and $l$ are integers satisfying: $0\leq s<k$ and $0\leq t < l$. Hence, using that $A$ commutes with its resolvent we have 
\[
A^sA^t = (A^kA^{s-k})(A^lA^{t-l}) = (A^{k+l})(A^{s+t-(k+l)})=A^{s+t}. 
\]
Let us continue with the inverse of $A^{s}$. First, let us assume $s$ is just a positive integer. In this case we have the equality $(re^{i\pi})^{-s}=(re^{-i\pi})^{-s}$, which will imply that the integrals over $\Gamma_1$ and $\Gamma_3$ cancel each other  out because they have opposite orientation. Indeed, we have 
\begin{align*}
    A^{-s} &= \frac{1}{2\pi i}\int_{\Gamma} \lambda^{-s} (A-\lambda I)^{-1}\, d\lambda = \frac{1}{2\pi i}\int_{\Gamma_1\cup\Gamma_2\cup\Gamma_3} \lambda^{-s} (A-\lambda I)^{-1}\, d\lambda\\ 
    &= \frac{1}{2\pi i}\left(\int_{\Gamma_1} \lambda^{-s} (A-\lambda I)^{-1}\, d\lambda- \int_{\Gamma_1} \lambda^{-s} (A-\lambda I)^{-1}\, d\lambda+\int_{\Gamma_2} \lambda^{-s} (A-\lambda I)^{-1}\, d\lambda\right)\\
    &=\frac{1}{2\pi i}\int_{\Gamma_2} \lambda^{-s} (A-\lambda I)^{-1}\, d\lambda = -\frac{1}{2\pi i}\int_{\tilde{\Gamma}_2} \mu^s \left(A-\frac{1}{\mu} I\right)^{-1}\, \frac{1}{\mu^2}\, d\mu \\
    &= \frac{A^{-1}}{2\pi i} \int_{\tilde{\Gamma}_2} \mu^{s-1} \left(A^{-1}-\mu I\right)^{-1}\, d\mu= A^{-1}(A^{-1})^{s-1} = (A^{-1})^{s},
\end{align*}
where $\tilde{\Gamma}_2=\{|\mu|=1/\epsilon\}$, and thus proving that $A^{-s}$ is the inverse of $A^s$. Therefore, together with previous group property, we have that for any $s\in \bbC$, $A^sA^{-s}=I$. Moreover, this also implies that $A^s$ is an isomorphism from $H^k$ to $H^{k-m\operatorname{Re}(s)}$. 

Finally, for analiticity, because \eqref{functionalCalculus} is convergent, we can just take the derivative inside the integral and obtain 
\[
\frac{d}{ds} A^s = \frac{1}{2\pi i } \int_\Gamma \lambda^s \log\lambda (A-\lambda I)^{-1}\, d\lambda,
\]
which converges in operator norm because of the same reasons \eqref{functionalCalculus} does. Then we can rewrite $A^s=A^{s-\nu/m}A^{\nu/m}$ giving us the boundedness from $H^k$ to $H^{k-\nu}$ for $\operatorname{Re}(s)<\nu$. 
\end{proof}
\begin{rem}
    We would like to mention that the last result includes a very useful operator, namely the square root $\sqrt{a(X,D)}$ just by taking $s=1/2$. 
\end{rem}
The next example will take advantage of the freedom given by Theorem \ref{thmFunctionalCalZero} for zero order pseudo-differential operators, and follows immediately from it. This one later on will be used for the Szeg\"o limit theorem.  
\begin{cor}\label{corLog}
Let $0\leq \delta<\rho \leq 1$. Let at least one of the conditions (1)-(3) of Theorem \ref{thmFunctionalCal} be fulfilled. Let $a\in S_{\rho, \delta}^{0}(\nabla)$ and let $\Lambda$ be the sector enclosed by a finite curve $\Gamma$ surrounding the spectrum of $a(X,D)$. If the symbol $a$ satisfies 
    \[
    \exp (\log( a)) \equiv a,
    \]
    then $\sigma_{\log a}(y,\eta)$ defines a symbol in $S_{\rho, \delta}^{0}(\nabla)$ and satisfies 
\begin{equation}\label{asymLog}
    \sigma_{\log a}(y,\eta) \sim   \log(a(x,\zeta))  
        + \sum_{k=1}^\infty \sum_{\mathfrak{I}_k}\frac{(-1)^k}{k+|\mathfrak{I}_{k,L}|} a(x,\zeta)^{-k-|\mathfrak{I}_{k,L}|} \mathfrak{r}_{\mathfrak{I}_k}(y,\eta)
\end{equation}
as $\langle\eta\rangle_{y}\to\infty$.       
\end{cor}
\begin{rem}\label{remNegativeLog}
    We emphasize that for any operator in the class $\Psi^{<0}$, the functional calculus developed for zero-order operators can still be applied. This is due to the inclusion $\Psi^{<0}\subseteq \Psi^{0}$, which ensures that the functional calculus yields a new pseudo-differential operator of order zero. However, in certain cases, it is possible to recover the original negative order of the operator; an important example being the logarithm. Specifically, let $\mathfrak{n}<0$ and suppose $a\in S_{\rho, \delta}^{\mathfrak{n}}(\nabla)$. Then, by Corollary \ref{corLog}, the symbol $\sigma_{\log a}$ belongs to $S_{\rho, \delta}^{0}(\nabla)$ and admits the asymptotic expansion \eqref{asymLog}. By carefully analyzing this expansion in conjunction with the original symbol estimates, one can show that in fact $\sigma_{\log a}\in S_{\rho, \delta}^{\mathfrak{n}}(\nabla)$, thereby recovering the negative order.
\end{rem}
Another classical and extensively studied case is the exponential of an operator, primarily due to its deep connection with time-evolution equations. However, when dealing with such operators, a slight modification in the choice of the contour $\Gamma$ is required. Aside from this adjustment, all previously established results apply mutatis mutandis. We now define a new contour $\Theta = \Theta_1\cup\Theta_2\cup\Theta_3$, constructed so that every $\lambda$ on the contour has a positive real part.
\begin{equation*}\label{contorno3}
    \Theta = \begin{cases}
    re^{i\left(\frac{\pi}{2}-\varepsilon \right)} \,(+\infty>r>\epsilon)\text{ on } \Theta_1,\\
    \epsilon e^{i\theta}\, (\frac{\pi}{2}-\varepsilon>\theta>-\frac{\pi}{2}+\varepsilon) \text{ on } \Theta_2,\\
    re^{-i\left(\frac{\pi}{2}-\varepsilon \right)}\, (\epsilon<r<+\infty)\text{ on } \Theta_3.
\end{cases}
\end{equation*} 
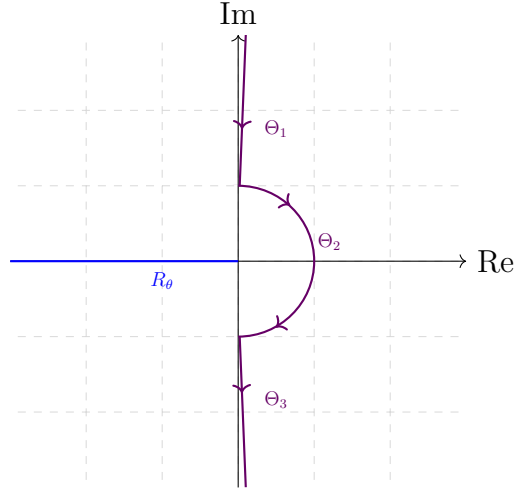
\begin{center}
    \begin{tikzpicture}
    \draw[help lines, color=gray!30, dashed] (-2.9,-2.9) grid (2.9,2.9);
\draw [->] (-3,0)--(3,0) node[right]{$\text{Re}$};
\draw [->] (0,-3)--(0,3) node[above]{$\text{Im}$};
  \draw[-,blue, thick] (-3,0)--(0,0) node[xshift=-1cm, below]{$\resizebox{.02\hsize}{!}{$R_\theta$}$};

    % draw the two lines
    \begin{scope}[even odd rule, decoration={
    markings,
    mark=at position 0.5 with {\arrow{>}}}]
        \clip (2,-3) rectangle (-4,3) (0,0) circle (1);
        \draw [postaction={decorate}, color={rgb,255:red,102; green,0; blue,102}, thick] (0.1,3) -- (0,0.5) node[xshift=0.5cm, yshift=1cm, above]{$ \resizebox{.02\hsize}{!}{$\Theta_1$} $};
        \draw [postaction={decorate}, color={rgb,255:red,102; green,0; blue,102}, thick] (0,-0.5) -- (0.1,-3) node[xshift=0.4cm, yshift=0.9cm,  above]{$\resizebox{.02\hsize}{!}{$\Theta_3$}$};
        
    \end{scope}

    % draw the circle
    \begin{scope}[even odd rule, decoration={
    markings,
    mark=at position 0.15 with {\arrow{<}},
    mark=at position 0.85 with {\arrow{<}}}]
        \clip (-2,-2) rectangle (3,2) (0,1.02) rectangle (-3,-1.02);
        \draw [postaction={decorate}, color={rgb,255:red,102; green,0; blue,102}, thick] (0,0) circle (1) node[xshift=1.2cm, above]{$\resizebox{.02\hsize}{!}{$\Theta_2$}$};
    \end{scope}
\end{tikzpicture}
\end{center}
Thus we define $\Sigma$ to be the sector surrounded by $\Theta$. 
\begin{rem}
    Actually, $\Theta$ is just a nice deformation of $\Gamma$ so that the real part of the elements of the curve has positive real part. 
\end{rem}
\begin{cor}\label{exponAsym}
     Let $0\leq \delta<\rho \leq 1$, $m\geq 0$. Let at least one of the conditions (1)-(3) of Theorem \ref{realEstiParameSymbol} be fulfilled. Let $a\in S_{\rho, \delta}^{m}(\nabla)$ be parameter-elliptic with respect to $\Sigma$ and assume that the operator $a(X,D)$ is positive real (so that its spectrum has positive real part). Then $\sigma_{e^{-a}}(y,\eta)$ defines a symbol in $S_{\rho, \delta}^{-m}(\nabla)$ and satisfies
    \[
    \sigma_{e^{-a}}(y,\eta) \sim e^{-a(y,\eta)}\left[1 +   \sum_{k=1}^\infty \sum_{\mathfrak{I}_k} \frac{(-1)^{1+k}}{(k+|\mathfrak{I}_{k,L}|)!}\mathfrak{r}_{\mathfrak{I}_k}(y,\eta)\right],
    \]
as $\langle\eta\rangle_{y}\to\infty$. Furthermore, if $m>0$ the exponential operator $e^{-a(X,D)}$ is actually smoothing. 
\end{cor}
\begin{proof}
Let $\bbC^{+}:= \{z\in\bbC : \operatorname{Re(z)>0}\}$. Note that the function $f: \bbC^{+} \to \bbC^{+}$ given by $f(z)=e^{-z}$ will not satisfy the estimate \eqref{decaiminetoFuncion}, but 
\[
|f(z)| = |e^{-z}| = e^{-\operatorname{Re}(z)} \leq \frac{1}{1+\operatorname{Re}(z)}.
\]
So using this estimate we can give sense to the integral 
\[
e^{-a(X,D)} = \frac{1}{2\pi i}\int_{\Theta} e^{-z} (A-z I)^{-1}\, dz,
\]
and thus obtain an analogue for Theorem \ref{thmFunctionalCal}. Moreover, for any $N>1$ a much stronger inequality holds, namely 
\[
|f(z)| \leq \frac{1}{(1+\operatorname{Re(z)})^N},
\]
so that if $m>0$ we get that the exponential operator is smoothing. 
\end{proof}

\section{Applications}\label{Section 5}
We conclude the article with the applications of our main result mentioned in the introduction, which essentially consists of calculating traces of some functions of operators. Here $(M,g)$ is still a closed Riemannian manifold. 

\subsection{Szeg\"o limit theorem}

Our first application is a generalization of Szeg\"o theorem in the sense of Widom \cite{w2}. Indeed this generalizes Widom's result because we are not just considering $(1,0)$-classes, but $(\rho,\delta)$-classes. Although, we are not considering families of operators. 

In Subsection \ref{traceSubsection} we defined a trace for Safarov pseudo-differential operators introducing us the notion of trace-class operators. For such operators we are going to define the so called \textit{log-determinant}: 

\begin{defn}
    Let $0\leq \delta<\rho \leq 1$ and $-\mathfrak{n}<-n$. Let $a\in S_{\rho, \delta}^{-\mathfrak{n}}(\nabla)$, then the determinant of $I+a(X,D)$ is defined as 
    \[
    \log \det(I+a(X,D)) = \exp\left\{\tr(\log(I+a(X,D)))\right\}. 
    \]
\end{defn}
 
\begin{thm}[Szeg\"o type-theorem]\label{LimitThe}
    Let $0\leq \delta<\rho \leq 1$ and $\mathfrak{n}>n$. Let at least one of the conditions (1)-(3) of Theorem \ref{thmFunctionalCal} be fulfilled. Let $a\in S_{\rho, \delta}^{-\mathfrak{n}}(\nabla)$ be such that 
    \[
    \exp (\log( 1+a)) \equiv 1+a. 
    \]
    Then $\det(I+a(X,D))\neq 0$ and one has the equality
    \begin{align*}
        \log\det(1+a(X,D)) = \frac{1}{(2\pi)^n}&\left[ \int_{T^*M} \log(1+a(x,\zeta)) \, \Omega \right.\\
        &+ \left.\sum_{k=1}^\infty \sum_{\mathfrak{I}_k}\frac{(-1)^k}{k+|\mathfrak{I}_{k,L}|}\int_{T^*M} (1+a(x,\zeta))^{-k-|\mathfrak{I}_{k,L}|} \mathfrak{r}_{\mathfrak{I}_k}(x,\zeta) \, \Omega\right],
    \end{align*}
    where $\Omega$ is as in \eqref{cotangentBundleForm}.
\end{thm}
\begin{proof}
Due to Remark \ref{remNegativeLog} the new operator $\log(I+a(X,D))$ is again of order $-\mathfrak{n}$, so that it is a trace-class pseudo-differential operator and its determinant is well-defined. Furthermore, Theorem \ref{traza} and the asymptotic formula \eqref{asymptForFunctionOfSymbol} allow us to calculate such determinant as follows:  
\begin{align*}
        \log\det(1+a(X,D)) = \frac{1}{(2\pi)^n}&\left[ \int_{T^*M} \log(1+a(x,\zeta)) \, \Omega \right.\\
        &+ \left.\sum_{k=1}^\infty \sum_{\mathfrak{I}_k}\frac{(-1)^k}{k+|\mathfrak{I}_{k,L}|}\int_{T^*M} (1+a(x,\zeta))^{-k-|\mathfrak{I}_{k,L}|} \mathfrak{r}_{\mathfrak{I}_k}(x,\zeta) \, \Omega\right],
    \end{align*}
    finishing the proof. 
\end{proof}

\subsection{Traces of heat operators} First, let us deal with general exponential operators. Notice that we can obtain expressions for the kernel and trace of such operators just by putting together Proposition \ref{kernelEnxx}, Theorem \ref{traza} and Corollary \ref{exponAsym}, this is stated in the following theorem.
\begin{thm}\label{ExponTrace}
   Let $0\leq \delta<\rho \leq 1$, $m\geq 0$. Let at least one of the conditions (1)-(3) of Theorem \ref{thmFunctionalCal} be fulfilled. Let $a\in S_{\rho, \delta}^{m}(\nabla)$ be parameter-elliptic with respect to $\Sigma$ and assume that the operator $a(X,D)$ is positive (so that its spectrum has positive real part). Then the kernel of the operator $e^{-a(X,D)}$ satisfies 
    \[
    \mathscr{A}(x,x) = \frac{1}{(2\pi)^n} \int_{T_x^*M} e^{-a(x,\zeta)}\left[1 +   \sum_{k=1}^\infty \sum_{\mathfrak{I}_k} \frac{(-1)^{1+k}}{(k+|\mathfrak{I}_{k,L}|)!}\mathfrak{r}_{\mathfrak{I}_k}(x,\zeta)\right]\, d\zeta.
    \]
     Moreover
    \[
    \tr\left( e^{-a(X,D)}\right) = \frac{1}{(2\pi)^n} \int_{T^*M} e^{-a(x,\zeta)}\left[1 +   \sum_{k=1}^\infty \sum_{\mathfrak{I}_k} \frac{(-1)^{1+k}}{(k+|\mathfrak{I}_{k,L}|)!}\mathfrak{r}_{\mathfrak{I}_k}(x,\zeta)\right]\, \Omega,
    \]
    where $\Omega$ is as in \eqref{cotangentBundleForm}.
\end{thm}
These results acquire special significance when considering operators of the form $ta(X,D)$, with $t>0$, as the semigroup $e^{-ta(X,D)}$ provides solutions to the heat equation:
\[
\begin{cases}
    \partial_t u(t,x)+a(X,D)u(t,x)=0,\quad t>0, \,x\in M,\\
    u(0,x)=u_0(x).
\end{cases}
\]
The constructions developed in the previous sections remain valid for the family $ta(X,D)$; however, particular care must be taken with the dependence on the parameter $t$. By construction, we have
\begin{align*}
    e^{-ta(X,D)} = \frac{1}{2\pi i}\int_{\Theta} e^{-z} (ta(X,D)-z I)^{-1}\, dz = \frac{1}{2\pi i}\int_{\Theta'} e^{-tw} (a(X,D)-w I)^{-1}\, dw, 
\end{align*}
where $\Theta'=t\Theta$. Observe that the approximation by the parameter-dependent parametrix $A_w^\sharp$ (as used in \eqref{approxResolventOperator}) is valid for sufficiently large $|w|$. Therefore, under the change of variables above, this approximation holds for sufficiently small $t$. Moreover, by the Cauchy integral theorem, we may deform the contour $\Theta'$ back to a fixed path $\Theta$, independent of $t$. Consequently, we obtain the representation:
\[
e^{-ta(X,D)} = \frac{1}{2\pi i}\int_{\Theta} e^{-tw}A_w^\sharp \, dw + \frac{1}{2\pi i}\int_{\Theta} e^{-tz} \left[(A-w I)^{-1}- A_w^\sharp\right]\, dw.
\]
As $t\to 0$, the second integral becomes negligible, due to estimates analogous to those in Corollary \ref{normResolvent}.

Finally, this leads to an asymptotic expansion for the heat operator:
\begin{thm}\label{heatKernelExpansion}
     Under same hypothesis as in Theorem \ref{ExponTrace} we have that in the $C^k$ norm  
    \[
    \mathscr{A}_t(x,x) \sim \frac{1}{(2\pi)^n} \int_{T_x^*M} e^{-t a(x,\zeta)}\left[1 +   \sum_{k=1}^\infty \sum_{\mathfrak{I}_k} \frac{(-1)^{1+k}}{(k+|\mathfrak{I}_{k,L}|)!}\mathfrak{r}_{\mathfrak{I}_k}(x,\zeta)\:t^{k+|\mathfrak{I}_{k,L}|}\right]\, d\zeta
    \]
    as $t\to 0$. That is, for any $k\in\bbN$ there exists $N_k\geq 0$ such that for any $N \geq N_k$ it holds
    \begin{align*}
        \left\|\mathscr{A}_t(x,x) - \frac{1}{(2\pi)^n} \int_{T_x^*M} e^{-t a(x,\zeta)}\left[1 +   \sum_{k=1}^N \sum_{\mathfrak{I}_k} \frac{(-1)^{1+k}}{(k+|\mathfrak{I}_{k,L}|)!}\mathfrak{r}_{\mathfrak{I}_k}(x,\zeta)\:t^{k+|\mathfrak{I}_{k,L}|}\right]\, d\zeta\right\|_k\\
        \leq C_k t^k
    \end{align*}
for some $C_k>0$. Moreover
    \[
    \tr\left( e^{-ta(X,D)}\right) \sim \frac{1}{(2\pi)^n} \int_{T^*M} e^{-t a(x,\zeta)}\left[1 +   \sum_{k=1}^\infty \sum_{\mathfrak{I}_k} \frac{(-1)^{1+k}}{(k+|\mathfrak{I}_{k,L}|)!}\mathfrak{r}_{\mathfrak{I}_k}(x,\zeta)\:t^{k+|\mathfrak{I}_{k,L}|}\right]\, \Omega
    \]
    as $t\to 0$, where $\Omega$ is as in \eqref{cotangentBundleForm}. That is, for any $k\in\bbN$ there exists $N_k\geq 0$ such that for any $N \geq N_k$ it holds
\begin{align*}
    \left\|\tr\left( e^{-ta(X,D)}\right) - \frac{1}{(2\pi)^n} \int_{T^*M} e^{-t a(x,\zeta)}\left[1 +   \sum_{k=1}^N \sum_{\mathfrak{I}_k} \frac{(-1)^{1+k}}{(k+|\mathfrak{I}_{k,L}|)!}\mathfrak{r}_{\mathfrak{I}_k}(x,\zeta)\:t^{k+|\mathfrak{I}_{k,L}|}\right]\, \Omega\right\|_k\\
    \leq C_k t^k
\end{align*}
for some $C_k>0$.
\end{thm}
\begin{proof}
We just need to proceed as in the final part of the proof of Theorem \ref{thmFunctionalCal} taking care where $t$ appears in the formuli. Hence as $t\to0$ and $\langle\zeta\rangle_{x}\to\infty$, using \eqref{asymPara2}, we have that 
\begin{align*}
    \sigma_{e^{-ta}}(x,\zeta) &\sim \frac{1}{2\pi i}\int_\Theta e^{-t w} a^\sharp(x,\zeta,w)\,dw\\
    &\sim \frac{1}{2\pi i}\int_\Theta e^{-t w} \Bigg((a(x,\zeta)-w)^{-1} \\ 
    &+ \left. \sum_{k=1}^\infty \sum_{\mathfrak{I}_k}(-1)^k  \mathfrak{r}_{\mathfrak{I}_k}(x,\zeta)(a(x,\zeta)-\lambda)^{-1-k-|\mathfrak{I}_{k,L}|}\right)\,dw\\
    & \sim e^{-t a(x,\zeta)}\left[1 +   \sum_{k=1}^\infty \sum_{\mathfrak{I}_k} \frac{(-1)^{1+k}}{(k+|\mathfrak{I}_{k,L}|)!}\mathfrak{r}_{\mathfrak{I}_k}(x,\zeta)\:t^{k+|\mathfrak{I}_{k,L}|}\right].
\end{align*}
Now the result follows from Proposition \ref{kernelEnxx}, Theorem \ref{traza} and the estimates from Corollary \ref{normResolvent}.  
\end{proof}
\subsection{Spectral \texorpdfstring{$\zeta$}{L}-functions}
We finish this article by calculating traces of certain complex powers of Safarov pseudo-differential operators, this procedure is usually known as defining the spectral $\zeta$-functions associated to a given operator $A$. 
\begin{defn}
     Let $0\leq \delta<\rho \leq 1$ and $m>0$. Let $a\in S_{\rho, \delta}^{m}(\nabla)$ be parameter-elliptic with respect to $\Lambda$. If the symbol $a$ satisfies 
    \[
    \exp (\log( a)) \equiv a,
    \]
    we define the associated spectral $\zeta$-function for $\operatorname{Re}(z)m>n$ as follows: 
    \[
    \zeta(a(X,D),-z):= \tr (a(X,D)^{-z}).
    \]
\end{defn}
Therefore, using Corollary \ref{ComplexPow} and Theorem \ref{traza} we can immediately compute the spectral functions obtaining the following result: 
\begin{thm}\label{zetaFun}
    Let $0\leq \delta<\rho \leq 1$, $m> 0$. Let at least one of the conditions (1)-(3) of Theorem \ref{thmFunctionalCal} be fulfilled. Let $a\in S_{\rho, \delta}^{m}(\nabla)$ be parameter-elliptic with respect to $\Sigma$. Then for $\operatorname{Re}(z)m>n$ we have 
    \begin{align*}
       &\zeta(a(X,D),-z) = \int_{T^*M} a(x, \zeta)^{-z} \, \Omega \,\\
       &+ \sum_{k=1}^\infty \sum_{\mathfrak{I}_k}(-1)^{1+|\mathfrak{I}_{k,L}|}(-z)(-z-1)\cdots(-z-k-|\mathfrak{I}_{k,L}|+1) \frac{a(y,\eta)^{-z-k-|\mathfrak{I}_{k,L}|}}{(k+|\mathfrak{I}_{k,L}|)!}\mathfrak{r}_{\mathfrak{I}_k}(x,\zeta) \, \Omega.
    \end{align*}
\end{thm}

\bibliography{mybib}{}
\bibliographystyle{habbrv}

\end{document}